\documentclass[11pt]{amsart}
\usepackage[dvipsnames,usenames]{color}
\usepackage{hyperref}
\usepackage{graphicx}
\usepackage{epsfig}
\usepackage[latin1]{inputenc}
\usepackage{amsmath}
\usepackage{amsfonts}
\usepackage{amssymb}
\usepackage{amsthm}
\usepackage{amscd}
\usepackage{verbatim}
\usepackage{subfigure}
\usepackage{pinlabel}
\usepackage{stmaryrd}
\usepackage{enumerate, enumitem}

\usepackage[textsize=scriptsize]{todonotes}

\usepackage{tikz}
\usetikzlibrary{cd}
\usetikzlibrary{arrows}
\usetikzlibrary{decorations.pathreplacing}
\usepackage{verbatim}

\newtheorem{theorem}{Theorem}[section]

\newtheorem{lemma}[theorem]{Lemma}
\newtheorem{proposition}[theorem]{Proposition}

\newtheorem{corollary}[theorem]{Corollary}

\theoremstyle{definition}
\newtheorem{definition}[theorem]{Definition}

\theoremstyle{remark}
\newtheorem{remark}[theorem]{Remark}

\def\F{\mathbb{F}}
\def\N{\mathbb{N}}

\def\Z{\mathbb{Z}}

\def\bbT{\mathbb{T}}

\def\X{\mathbb{X}}

\def\cC{\mathcal{C}}

\def\II {\mathcal{I}}
\def\JJ {\mathcal{J}}

\def\cC{\mathcal{C}}

\def\cH{\mathcal{H}}

\def\cR{\mathcal R}

\def\bfalpha{\boldsymbol{\boldsymbol{\alpha}}}
\def\bfbeta{\boldsymbol{\boldsymbol{\beta}}}
\def\CFKUV{\CFK_\cR} 
\def\bfx{\mathbf{x}}
\def\bfy{\mathbf{y}}

\def\CF{\operatorname{CF}}

\def\CFh{\widehat{\operatorname{CF}}}
\def\HF{\operatorname{HF}}
\def\HFp{\operatorname{HF}^+}

\def\HFh{\widehat{\operatorname{HF}}}

\def\CFK{\operatorname{CFK}}
\def\CFKi{\operatorname{CFK}^{\infty}}

\def\CFKh{\widehat{\operatorname{CFK}}}

\def\HFKh{\widehat{\operatorname{HFK}}}

\def\Im{\operatorname{im}}

\def\d{\partial}
\def\varep{\varepsilon}
\def\co{\colon}

\def\J{\mathfrak{J}}
\def\K{\tilde{K}}
\def\CZhat{\widehat{\cC}_\Z}
\def\CZ{{\cC}_\Z}

\def\conn {\mathbin{\#}}

 \DeclareMathOperator{\gr}{gr}

\def\x {\mathcal{x}}

\def\U {\mathcal{U}}
\def\V {\mathcal{V}}
\def\C{\mathfrak{C}}
\def\D{\mathfrak{D}}

\definecolor{darkgreen}{rgb}{0,0.5,0}
\definecolor{purple}{rgb}{0.5,0,0.5}

\usepackage{lipsum}

\newcommand\blfootnote[1]{%
  \begingroup
  \renewcommand\thefootnote{}\footnote{#1}%
  \addtocounter{footnote}{-1}%
  \endgroup
}

\author[]{Hugo Zhou}
\address {School of Mathematics, Georgia Institute of Technology, Atlanta, GA 30332}
\email{hzhou92@gatech.edu}

\title{Homology concordance and An Infinite rank free subgroup}
\begin{document}
\maketitle
\makeatletter
\def\blfootnote{\gdef\@thefnmark{}\@footnotetext}
\makeatother
\blfootnote{The author was partially supported by NSF grant DMS-1552285.}

\begin{abstract}
Two knots are homology concordant if they are smoothly concordant in a homology cobordism. The group $\CZhat$ (resp. $\CZ$) was previously defined as the set of knots in homology spheres that bounds homology balls (resp. in $S^3$), modulo homology concordance. We prove  $\CZhat / \CZ$ contains a $\Z^{\infty}$ subgroup. We construct our family of examples by applying the filtered mapping cone formula to L-space knots, and prove linear independence with the help of the connected knot complex.
\end{abstract}

\section{Introduction}
\label{sec:intro_paper}

Two knots in $S^{3}$ are {\em smoothly concordant} if they cobound a smoothly embedded annulus in the product $S^{3}\times I$. One can consider the set of knots in $S^{3}$ modulo smooth concordance. This set becomes an abelian group called the {\em knot concordance group}, under the operation induced by connected sum. It is a well-known result that the knot concordance group contains a $\Z^{\infty}$ summand (see, for example, \cite{LevKnot}).

As a generalization of smooth concordance, the notion {\em homology concordance} is  studied in \cite{DR13} and \cite{HomLevineLidman}. This paper will build on the established results in \cite{HomLevineLidman} regarding the homology concordance classes of knots in integer homology spheres.

Given two knots $K \subset Y$, $K' \subset Y'$, where both $Y$ and $Y'$ are integer homology spheres that are homology cobordant to $S^{3}$, we call $K$ and $K'$  {\em homology concordant} if they cobound a smoothly embedded annulus in some homology cobordism between $Y$ and $Y'$, denoted $(Y, K) \sim (Y', K')$.

Let $\CZhat$ denote the group consists of pairs $(Y,K)$, where $Y$ is a homology null-cobordant homology $3$-sphere and $K$ a knot in $Y$, modulo homology concordance. The operation, still induced by connected sum now changes both the knot and the $3$-manifold: $(Y,K) \conn (Y',K')$ is given by $(Y \conn Y',K \conn K')$. The inverse of $(Y,K)$ is $(-Y,-K)$. 

Similarly, let $\CZ$ denote the group consists of pairs $(S^{3},K)$ modulo homology concordance, which is a natural subgroup of $\CZhat$.

One wants to study the group structure of $\CZhat$, since this is the group analogous to the knot concordance group for knots in homology spheres. Note that $\CZhat$ naturally inherits the complicated group structure from the knot concordance group. In fact, there is a canonical map from the knot concordance group to $\CZ$. Though whether this map is injective remains unknown, J. Levine's examples pass through and constitute a $\Z^\infty$ summand in $\CZ$ (For a construction of Levine's examples, see \cite{LevKnot} or \cite{LivinSurvey}). Due to this observation, we consider instead $\CZhat / \CZ $.  The quotient group $\CZhat / \CZ $ measures the ``difference''  between knots in $S^3$ and knots in arbitrary homology spheres that bound homology balls.  Adam Simon Levine showed in \cite{Lev14} that $\CZhat / \CZ $ is not trivial. Inspired by the study of the knot concordance group, we ask: Is $\CZhat / \CZ $ infinitely generated? Does it contain a $\Z^{\infty}$ subgroup?

 The first question has been given an affirmative answer in \cite{HomLevineLidman}. Moreover, two approaches to the question have been demonstrated. The first approach relies on an invariant $\theta (Y,K)$, defined as the maximal difference of $d$-invariants of $1/n$-surgeries on $K$ (For a definition of $d$-invariant see  \cite{OSAbso}). A family of pairs $(Y_j,K_j)$ was constructed whose $\theta (Y_j, K_j) \rightarrow \infty$ as $j$ approaches infinity. With some more argument, this shows the infinite generation of $\CZhat / \CZ $. However, it is unknown to the author at this point whether the pairs constructed in the proof have infinite order. Another approach is to generalize the knot Floer invariants including $\tau, \nu, \varep, \Upsilon$ which originally were only defined for knots in $S^3$, to homology concordance invariants.  An advantage of this approach is that many of these invariants still have the same additivity properties under connected sums of the knots in an arbitrary integer homology sphere as they do for knots in $S^{3}$. We will adopt this approach in the paper, and answer the second question affirmatively, as follows.
 
 \begin{theorem}\label{thm:main}
 $\CZhat / \CZ $ contains a $\Z^{\infty}$ subgroup.
 \end{theorem}

One of the main tools we use in the proof is the following concordance invariant. The invariant is more or less known to the experts. Since its definition has not yet appeared in the literature, we write it down for completeness.

\begin{theorem}\label{thm:cconn}
Given a knot $K$ in an integer homology sphere $Y$, suppose $C=CFK^{\infty}(Y,K)$ is the knot Floer complex. There is a homology concordance invariant $C_{\mathrm{conn}}$, taking value in isomorphism classes of absolutely-graded, $(i,j)$-filtered $\F[U,U^{-1}]$  chain complexes. Moreover, $C$ is filtered chain homotopy equivalent to $C_{\mathrm{conn}}\oplus A$ as a filtered chain complex, where $A$ is an acyclic complex.
\end{theorem}   

The invariant is defined in the same way as the connected Heegaard Floer homology in \cite{HHL}, ignoring the involution in the original definition. We describe the construction in Section $6$ over the ring $\F[\U,\V]$. The same construction works over the ring $\F[U,U^{-1}]$ as well.

Regarding $\tau$ and $\varep$ invariants, we review the definitions and properties in section \ref{sec:con}. The key observation is:

\begin{proposition}[Proposition 3.6(2) in \cite{Hom12}]\label{prop:tauep}
Suppose $K$ is a knot in $S^3$. If $\varep(S^3,K)=0$, then $\tau(S^3,K)=0$.
\end{proposition}

We construct the $\Z^{\infty}$ subgroup by explicitly finding infinitely many generating pairs. For $n>1$, let $M_n$ denote $+1$-surgery on $T_{2,4n-1}$, the $(2,4n-1)$-torus knot, let $Y_n=M_n\conn -M_n$, and let $K_n\subset Y_n$ denote the connected sum of the core of the surgery in $M_n$ with the unknot in $-M_n$. 

\begin{proposition}\label{prop:kn}
For $n>1$ and $(Y_n,K_n)$ as above, we have the following,
\begin{enumerate}
\item \label{item:kntau}$\tau(Y_n,K_n)=-1$; 
\item \label{item:knep}$\varep(Y_n,K_n)=0$;
\item \label{item:kninde}$(Y_n,K_n)$ are linearly independent in $\CZhat  / \CZ $.
\end{enumerate}
\end{proposition}

 Theorem \ref{thm:main} directly follows from Proposition \ref{prop:kn} \eqref{item:kninde}. The construction of $(Y_n,K_n)$ and the evaluation of their $\tau$ and $\varep$ invariants are due to a computational result, Theorem \ref{thm:Lspace}.   
 
 \begin{remark}
 The quotient group $\CZhat/\CZ$ can also be interpreted in PL(piecewise-linear) terms. Recall that every knot in $S^{3}$ bounds a PL disk in $B^4$ by taking the cone. Such a disk is smooth outside the cone point.  In \cite{Lev14}, Adam Simon Levine gives a pair $(Y,K)$ such that $Y$ bounds a contractible manifold but $K$ does not bound a PL disk in any homology ball whose boundary is $Y$. Note that a knot $K$ in $Y$ bounds a PL disk if and only if $(Y,K)$ is trivial in $\CZhat/ \CZ$. One can see this by adding or deleting open balls near the cone points. Likewise, two
pairs $(Y_0, K_0),(Y_1, K_1) \in \CZhat$ differ by an element of $\CZ$ if and only if $K_0$ and $K_1$ cobound a PL
annulus in some homology cobordism from $Y_0$ to $Y_1$. Thus, the quotient $\CZhat/\CZ$ can be interpreted
as the group of knots in homology null-cobordant homology spheres modulo PL concordance in
homology cobordisms.
 \end{remark}

 \begin{remark}
 In the paper \cite{HomLevineLidman} (Remark 1.13.), a variation of $\CZhat$ was brought into discussion. Let $\CZhat'$ denote the subgroup of $\CZhat$ that consists of all pairs $(Y,K)$ where $Y$ bounds a homology ball $X$ in which $K$ is freely nulhomotopic (meaning $K$ is trivial in $\pi_1(X)$, or equivalently $K$ bounds a immersed disk in $X$). From the $4$-manifold perspective, $\CZhat'$ is arguably a more appropriate generalization of the knot concordance group, as it measures the failure of replacing immersed disks by embedded ones. It has been shown in \cite{Daemi} that $\CZhat'$ is a proper subgroup of $\CZhat$, answering a question raised up in \cite{HomLevineLidman}. Moreover, according to Corollary $1$ of \cite{Daemi}, each pair $(Y_n,K_n)$ in our construction lies in  $\CZhat \backslash \CZhat'$ since the core of the surgery solid torus normally generates the foundamental group of $+1$-surgery on $T_{2,4n-1}$. We do not know whether their linear combinations lie in $\CZhat'$ or not.
 \end{remark}

One can also consider the analogues of $\CZ$ and $\CZhat$ in the topological category. Namely, we say knots $(Y_0,K_0)$ and $(Y_1,K_1)$ are \emph{topologically homology concordant} if they cobound a locally flat embedded annulus in a topological homology cobordism between $Y_0$ and $Y_1$ (which need not carry any smooth structure), and we let $\cC_{\Z,\mathrm{top}}$ and $\widehat\cC_{\Z,\mathrm{top}}$ denote the corresponding concordance groups, as in \cite{DR13}. Note that every homology $3$-sphere bounds a contractible topological $4$-manifold, so there is no need to impose restriction on which pairs $(Y,K)$ are represented in $\widehat\cC_{\Z,\mathrm{top}}$ as we did in smooth case (Recall that we required the homology spheres to bound homology balls in the smooth version). There is strong evidence that the natural inclusion $\varphi\co \cC_{\Z,\mathrm{top}} \to \widehat\cC_{\Z,\mathrm{top}}$ is an isomorphism (see \cite{Davis}).

Define the natural map $\psi \co \CZhat \rightarrow \widehat\cC_{\Z,\mathrm{top}} $ by simply forgetting the smooth structure. The kernel of the map $\psi$ consists of the pairs that are topologically homology concordant to $(S^{3},O)$, where $O$ is the unknot in $S^{3}$. We can refine our construction of the infinite generating pairs such that they are in $\operatorname{ker} \psi$. The following result is suggested by JungHwan Park.

\begin{theorem}\label{thm:Jung}
For $n>1$, there exists $(Y'_n,J_n)\in \CZhat$ such that
\begin{enumerate}
\item $(Y'_n,J_n)\in \operatorname{ker} \psi$;
\item $(Y'_n,J_n)$ are linearly independent in $\CZhat/\CZ$.
 \end{enumerate}
\end{theorem}

The proof of Theorem \ref{thm:Jung} is in Section 3.

\subsection*{Organization}
In Section $2$ we review invariants $\tau$ and $\varep$, generalized as concordance invariants. In Section $3, 4$ and $5$ we present a computational result of core of surgery on L-space knots using the filtered mapping cone formula by Hedden and Levine \cite{HeddenLevine}, which will in turn give the proof of Proposition \ref{prop:kn} (\ref{item:kntau}),(\ref{item:knep}). In Section $6$ we give the definition of the connected knot complex over the ring $\F[\U,\V]$. In Section $7$ we prove the linear independence of the family $(Y_n,K_n)$.

\subsection*{Acknowledgements}
The author wants to thank his advisor Jennifer Hom for her patient guidance throughout the project. The author is grateful to JungHwan Park for explaining Theorem \ref{thm:Jung} and would like to thank Tye Lidman, 
Adam Simon Levine and Chuck Livingston for helpful comments. The author also thanks the referee for helpful comments.

\section{Concordance invariants from knot floer homology}
\label{sec:con}

In this section, we review the concordance invariants that will be useful for the construction, defined in \cite{HomLevineLidman} using tools from \cite{Zem16}. We assume that the reader is familiar with the knot Floer invariant package in $S^{3}$, see \cite{HomSurvey} for a survey on the topic. 

Recall that for a knot $K$ in an integer homology sphere $Y$, over the ring $\F[U]$ knot floer complex $C=\CFKi(Y,K)$ decomposes as a vector space $C=\bigoplus_{i,j\in \Z}C(i,j)$. Up to chain homotopy we can choose the filtered basis such that it is reduced, namely no differential preserves the $(i,j)$ coordinate. Let $X\subset \Z\oplus \Z$ be a set such that if $a,c\in X$ and $ a\prec b \prec c$, then $b \in X$, where $\prec$ is the natural partial order on $\Z\oplus \Z$. Let $CX = \bigoplus_{(i,j)\in X}C(i,j) $. $CX$ is naturally a subquotient complex of $C$. 

Consider the maps
\begin{align*}
	\iota_s &\co C \{ i=0, j \leq s \} \rightarrow C \{ i = 0 \}, \\
	v_s &\co C \{ \max(i, j-s)=0 \} \rightarrow C \{ i = 0 \}, \\
	v'_s &\co C\{ i = 0\} \rightarrow C\{ \min(i, j-s) = 0 \}, \\
\end{align*}
where $\iota_s$ is inclusion, $v_s$ consists of quotienting by $C\{i < 0, j=s\}$ followed by inclusion and $v'_s$ consists of quotienting by $C\{ i=0,  j < s\}$ followed by inclusion. Recall that $C\{i=0\} \simeq \CFh(Y)$ and $C\{i\geq 0\} \simeq \CF^+(Y)$. Also, let $\rho \co \CFh(Y) \rightarrow \CF^+(Y)$ denote inclusion.

\begin{definition}[Definition 4.2 in \cite{HomLevineLidman}] \label{def:taunu}
Let $K$ be a knot in an integer homology sphere $Y$. Define
\begin{align*}
	 \tau(Y,K) &= \min \{ s \mid \Im (\rho_* \circ \iota_{s*}) \cap U^N \HF^+(Y) \neq 0 \ \forall N \gg 0\} \\
	 \nu(Y,K) &= \min \{ s \mid \Im (\rho_* \circ v_{s*}) \cap U^N \HF^+(Y) \neq 0 \ \forall  N \gg 0\} \\
	  \nu'(Y,K) &= \max \{ s \mid v'_{s*}(x)\neq 0 \ \forall x \in \HFh(Y) \textup{ s.t. } \rho_*(x) \neq 0 \textup{ and } \rho_*(x) \in U^N \HF^+(Y) \ \forall N \gg 0  \}.
\end{align*}
\end{definition}

The definition of $\varep$ relies on the relation between $\tau,\nu$ and $\nu'$; we have the following lemma:

\begin{lemma} [Lemma 4.5 in \cite{HomLevineLidman}]
Let $K$ be a knot in an integer homology sphere $Y$. The following three cases are exhaustive and mutually exclusive:
\begin{itemize}
	\item $\nu(Y,K)=\tau(Y,K)+1$ and $\nu'(Y,K)=\tau(Y,K)$,
	\item $\nu(Y,K)=\tau(Y,K)$ and $\nu'(Y,K)=\tau(Y,K)-1$,
	\item $\nu(Y,K)=\tau(Y,K)$ and $\nu'(Y,K)=\tau(Y,K)$.
\end{itemize}
\end{lemma}

The three cases in the above lemma correspond to the three different values of $\varep$. 

\begin{definition} [Definition 4.6 in \cite{HomLevineLidman}]
Let $K$ be a knot in an integer homology sphere $Y$. Define
\[	\varep(Y,K) =
		\begin{cases}
			-1 \quad & \text{if } \nu(Y,K)=\tau(Y,K)+1, \\
			1 \quad & \text{if } \nu'(Y,K)=\tau(Y,K)-1, \\
			0 \quad & \text{otherwise}.
		\end{cases}
\]
\end{definition}

The following properties will be especially useful.
 
 \begin{proposition}[Propositions 4.7, 4.10, 4.11 in \cite{HomLevineLidman}] \label{prop:taueppro}
Both $\tau, \varep$ are homology concordance invariants. Moreover, given pairs $(Y, K)$ and $(Y', K')$ as group elements in $\CZhat$,  
\begin{enumerate}
	\item $\tau(-Y,K) = -\tau(Y,K)$.
	\item $\tau (Y \conn Y', K \conn K') = \tau(Y,K) + \tau(Y',K')$.
	\item \label{it:ep(-Y,K)} $\varep(-Y,K) = -\varep(Y,K)$.
	\item \label{it:epadd1} If $\varep(Y, K)=\varep(Y', K')$, then $\varep(Y \conn Y', K \conn K') = \varep(Y, K)$.
	\item \label{it:epadd2} If $\varep(Y, K)=0$, then $\varep(Y \conn Y', K \conn K') = \varep(Y', K')$.
\end{enumerate}
\end{proposition}

Lastly, we recall the following obstruction to knots in $\CZ$ observed in \cite{HomLevineLidman}:

\begin{corollary}
Let $(Y, K)$ be an element in $\CZhat$. If $\varep (Y,K)=0$, but $\tau (Y,K) \neq 0$, then $(Y,K)$ generates a $\Z$-subgroup in  $\CZhat / \CZ$.
\end{corollary}

\begin{proof}
According to Proposition \ref{prop:tauep}, $(Y,K)$ is not homology concordant to any knot in $S^{3}$. In other words, $(Y,K)$ represents a non-trivial element in  $\CZhat / \CZ$. Suppose $\tau(Y,K)=a$, where $a \neq 0$. Now by Proposition \ref{prop:taueppro}, $\varep(\conn_{n}(Y,K))=0$ while $\tau(\conn_{n}(Y,K))=na$, and this completes the proof.
\end{proof}

\section{The core of surgery on L-space knots}
\label{sec:Lspace_intro}
Starting from this section, we will present the computational result that leads to our construction of the $\Z^{\infty}$-subgroup. It relies on the knot Floer homology of the core of a Dehn surgery on L-space knots. We now explain the context and give a more precise description of the result.

A knot $K \subset S^3$ is called an {\em L-space knot} if the surgery $S^3_n(K)$ is an L-space for $n$ large enough. These knots have relatively simple knot Floer complexes that are determined by the Alexander polynomial. See Section 7 of \cite{HendricksManolescu} for a more detailed review regarding L-space knots. Since their notations give a certain convenience in representing the symmetry relation of knot Floer complex of L-space knots, we will borrow them in this section and throughout the computation in the next two sections.

In \cite{OSLens}, Ozsv\'ath and Szab\'o  proved that if $K$ is an L-space knot, then the Alexander polynomial of $K$ is of the form
$$ \Delta_K(t) = (-1)^m + \sum_{i=1}^m (-1)^{m-i}(t^{n_i} + t^{-n_i})$$
for a sequence of positive integers $0 < n_1 < n_2 < \dots < n_m.$ Here, $n_m = g(K)$ is the genus of $K$. Let $n(K) \geqslant 0$ be the quantity
$$ n(K):= n_m - n_{m-1} + \dots + (-1)^{m-2} n_2 + (-1)^{m-1} n_1.$$ Furthermore, let $\ell_s = n_s - n_{s-1}$. Since the unknot is the only L-space knot with $n(K) =0$ (whose dual knot is the unknot in $S^3$), we will only consider knots with $n(K) \geqslant 1$ for our computational result.

L-space knots have a knot Floer complex where differentials form a ``staircase'' .  The knot Floer complex consists of $2m+1$ generators $x_0, x^1_1, x^2_1,..., x^1_m,x^2_m$. The $(i,j)$-filtration of $x_0, x^1_m$ are $(0,0)$, $(-n(K),g(K)-n(K))$, respectively. The filtration of $x^t_{m-2s}$ and $x^t_{m-(2s+1)} $ differ only in $i$-filtration by $\ell_{m-2s}$ and the filtration of $x^t_{m-(2s+1)}$ and $x^t_{m-(2s+2)} $ differ only in $j$-filtration by $\ell_{m-(2s+1)}$.   If $m$ is odd,  the differentials are
\begin{align*}
\d(x_0) &= x_1^1+x_1^2 & \d(x_s^t) &= x_{s-1}^t+x_{s+1}^t \text{ for } s>0 \text{ even},\ t \in \{1,2\}
\end{align*}
\noindent whereas if $m$ is even, the differentials are
\begin{align*}
\d(x_1^t) &= x_0 + x_2^t & \d(x_s^t) &= x_{s-1}^t+x_{s+1}^t \text{ for } s>1 \text{ odd},\ t \in \{1,2\}.
\end{align*}
Observe that $n(K)$ is the total length of the horizontal arrows in the top half of the complex.

Let $K$ be an L-space knot in $S^{3}$. Consider $\CFKi(S^3_1(K),\K)$, where $\K$ is the core circle of the surgery solid torus in $1$-surgery on $K$. We are now ready to state the following computational result.

\newpage
\begin{theorem}\label{thm:Lspace}
 Assume $K$ is an L-space knot whose knot complex has $2m+1$ generators and $n(K)=n \geqslant 1$. If $m$ is odd, then $\CFKi(S^3_1(K),\K)$ is filtered chain homotopy equivalent to $C_n \oplus A$ as a filtered chain complex,  where $A$ is an acyclic complex, and the complex $C_n$ is generated by three generators whose $(i,j)$-filtrations are $(n-1,n),(0,0)$ and $(n,n-1)$ respectively, both generators in filtration $(n-1,n)$ and $(n,n-1)$ having an arrow going to filtration $(0,0).$

If $m$ is even, $\CFKi(S^3_1(K),\K)$ is homotopy equivalent to $C_0 \oplus A$ as a filtered chain complex, where $A$ is an acyclic complex, and the complex $C_0$ consists of one generator with $(i,j)$-filtration $(0,0)$.
\end{theorem}

\begin{remark}
 In fact, the complex $C_n$ we have described in Theorem \ref{thm:Lspace} and Figure \ref{fig:CFKinftycore} is the connected knot complex of $\CFKi(S^3_1(K),\K)$, for L-space knot $K$ whose knot complex has $2m+1$ generators where $m$ is odd and $n(K)=n$. Since we do not need the connected knot complex to explain the computational result, we will only formally define it in a later section. For the precise definition and more details see Section $6$.
 \end{remark}

\begin{figure}
\begin{tikzpicture}

	\begin{scope}[thin, black!20!white]
		\draw [<->] (-0.5, 0.5) -- (4, 0.5);
		\draw [<->] (0.5, -0.5) -- (0.5, 4);
	\end{scope}

	\filldraw (0.5, 0.5) circle (2pt) node[] (a){};
	\filldraw (3.5, 2.5) circle (2pt) node[] (b){};
	\filldraw (2.5, 3.5) circle (2pt) node[] (c){};	

	\draw [very thick, -] (b) -- (a);
	\draw [very thick, -] (c) -- (a);

	\node  at (3.7,0) {$n$};
	\node  at (2.3,0) {$n-1$};
	\node  at (-0.4,2.5) {$n-1$};
      \node  at (-0.1,3.5) {$n$}; 
      \node  at (3.8,2.9) {$(0)$};


	\draw [ thick, dashed] (b) -- (3.5,0.5);
	\draw [thick, dashed]  (b) -- (0.5,2.5);
	\draw [ thick, dashed] (c) -- (2.5,0.5);
	\draw [thick, dashed]  (c) -- (0.5,3.5);

\end{tikzpicture}
\caption{The filtered chain complex $C_n$, where the number in parenthesis marks the grading. For an L-space knot $K$ whose knot complex has $2m+1$ generators where $m$ is odd and $n(K)=n$, $\CFKi(S^3_1(K),\K)$ is filtered chain homotopy equivalent to $C_n \oplus A$, where $A$ is an acyclic complex. }
\label{fig:CFKinftycore}
\end{figure}
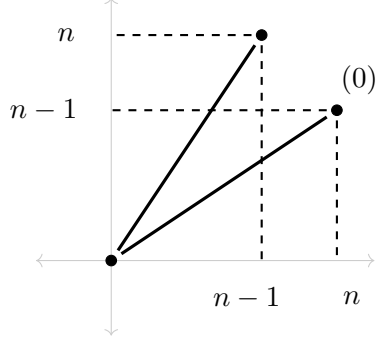

Before we move on to the method used in the proof of Theorem \ref{thm:Lspace}, we first explain why Theorem \ref{thm:Lspace} implies Proposition \ref{prop:kn} (\ref{item:kntau}), (\ref{item:knep}). For $n>1$, let $M_n$ denote $+1$-surgery on $T_{2,4n-1}$, the $(2,4n-1)$-torus knot, let $\tilde{T}_{2,4n-1}$ denote the core of the surgery in $M_n$, let $Y_n=M_n\conn -M_n$, and let $K_n\subset Y_n$ denote the connected sum of $\tilde{T}_{2,4n-1}$ in $M_n$ with the unknot in $-M_n$. 

\begin{proof}[Proof of Proposition \ref{prop:kn} (\ref{item:kntau}), (\ref{item:knep})]
The knot $T_{2,4n-1}$ has a symmetrized Alexander polynomial 
$$-1 + \sum_{i=1}^{2n-1} (-1)^{i-1}(t^{i} + t^{-i}).$$

In particular, $n(T_{2,4n-1})=n$ and there are $4n-1$ generators in the knot Floer complex. Applying Theorem \ref{thm:Lspace}, we see $\CFKi(M_n,\tilde{T}_{2,4n-1})$ is generated by 3 generators whose $(i,j)$-filtrations are $(-1,0),(0,-1)$ and $(-n,-n)$ respectively. Both generators in filtration $(-1,0)$ and $(0,-1)$ have an arrow going to filtration $(-n,-n).$ Consider the generator at filtration $(0,-1)$ and denote it by $\xi$. The homology class of $\xi$ in $\HF^+(M_n)$ is nontrivial in both $\Im (\rho_* \circ \iota_{-1*})$ and $\Im (\rho_* \circ v_{-1*})$ that is in the image of $U^N \HF^+(M_n)$ for all $N \gg
0$, and moreover $v'_{-1*}([\xi])$ is nontrivial. (See Definition \ref{def:taunu}.) 

From this observation, it is straightforward to verify that $\tau(M_n,\tilde{T}_{2,4n-1}) = \nu(M_n,\tilde{T}_{2,4n-1}) = \nu'(M_n,\tilde{T}_{2,4n-1}) = -1$, and so $\varep(M_n,\tilde{T}_{2,4n-1}) =0$. 
Connecting sum with the unknot in $-M_n$ doesn't change the value of $\tau$ or $\varep$, so it follows $\tau(Y_n,K_n)=-1$ and
$\varep(Y_n,K_n)=0$.
\end{proof}

\begin{remark}
By Theorem \ref{thm:Lspace}, if $K$ is an L-space knot in $S^3$, the only parameters that determine the connected knot complex of $\CFKi(S^3_1(K),\K)$ (See Definition \ref{def:conn}) are the value of $n(K)$ and the number of generators in $\CFKi(S^3,K)$. Hence our choice of knots is not particularly special. In fact, suppose we have a family of knots in $S^{3}$ such that for each knot, after a filtered change of basis, the knot Floer complex consists of acyclic summands and a positive staircase with $2m+1$ generators where $m$ is odd. Then as long as $n(K)$ of the staircase (the total length of the horizontal arrows in the top half of the staircase) is unbounded for this family of knots, they could be used to construct the $\mathbb{Z}^{\infty}$ subgroup.
\end{remark}

Next, we explain how to refine the construction such that the infinite generating pairs are inside $\operatorname{ker} \psi$, where $\psi$ is the natural map from $\CZhat $ to $ \widehat\cC_{\Z,\mathrm{top}} $. The following theorem is suggested to the author by JungHwan Park. The construction of the contractible $4$ manifold where the core of surgery bounds a slice disk in the proof is described to the author by Chuck Livingston.

First, we need a result from Hedden, Kim and Livingston's work \cite{MR3466802}. Let $D$ denote the positively-clasped untwisted Whitehead double of the right-handed trefoil, $T_{2,3},$ and let $D_k$ denote $kD$. 

\begin{proposition}[Proposition 6.1 in \cite{MR3466802}]
The chain complex $\CFKi(S^{3},D_k)$ is filtered chain homotopy equivalent to the chain complex $\CFKi(S^{3},T_{2,2k+1})\oplus A,$ where $A$ is an acyclic complex.
\end{proposition} 

Since $D$ has Alexander polynomial one, by Freedman and Quinn's work \cite{FreQuin}, $D$ is topologically slice and so is $D_k$. 

\begin{proof}[Proof of Theorem \ref{thm:Jung}.]
 For $n>1$, set $M'_n = S^3_{+1}(D_{2n-1})$, set $Y'_n=M'_n\conn -M'_n$, let $\tilde{D}_{2n-1}$ denote the core of the surgery in $M'_n$ , and let $J_n\subset Y'_n$ be the connected sum of $\tilde{D}_{2n-1}$ with the unknot in $-M'_n$. Since $\CFKi(S^{3},D_{2n-1})$ is filtered chain homotopy equivalent to the chain complex $\CFKi(S^{3},T_{2,4n-1})\oplus A,$  where $A$ is an acyclic complex, chosing $(Y'_n,J_n)$ provides us the same family of knot complexes $C_n$ as  $(Y_n,K_n)$ according to Theorem \ref{thm:Lspace}. 
 
 It remains to prove $(Y'_n,J_n)$ is topologically homology concordant to $(S^{3},O)$, where $O$ is the unknot in $S^{3}$. In fact, $M'_n = S^3_{+1}(D_{2n-1})$ bounds a contractible $4$ manifold where $\tilde{D}_{2n-1}$ bounds a locally flat disk. The construction is as follows. Attach a $2$-handle to $B^4$ along $D_{2n-1}$ with framing $1$ to form $W_n$. There is a $2$-sphere inside $W_n$ built from the core of the $2$ handle and the slice disk of $D_{2n-1}$. Since the framing of the attaching $2$-handle is $1$, the normal bundle of this $2$-sphere has Euler number $1$ and its boundary is $S^3$. Cut out the neighbourhood of this $2$-sphere and replace it with $B^4$, yielding a new $4$ manifold $Z_n$. The boundary of $Z_n$ is $S^3_{+1}(D_{2n-1})$ and $Z_n$ is contractible (To see this, first note that $Z_n$ is simply connected, and by Mayer-Vietoris sequence one can show it has the same homology as a point. Therefore $Z_n$ is contractible according to Whitehead Theorem). For $+1$ surgery, the core of surgery is isotopic to the zero framed longtitude in the complement of the knot, which bounds a locally flat disk in $B^4$ disjoint from the slicing disk for $D_{2n-1}$. Thus, $\tilde{D}_{2n-1}$ bounds a locally flat disk in $Z_n$. So $J_n$ bounds a topologically slice disk inside $Z_n \conn -Z_n$ and $(Y'_n,J_n)$ is topologically homology concordant to $(S^{3},O)$.

\end{proof}

The next two sections are dedicated to the proof of Theorem \ref{thm:Lspace}.

\section{The Filtered Mapping cone formula}
\label{sec:T35}

In this section, we briefly describe the filtered mapping cone formula and include an example of its application. We assume the reader is familiar with the surgery mapping cone formula by Ozsv\'ath and Szab\'o \cite{OS04}.

The filtered mapping cone formula from \cite{HeddenLevine}, is a refinement of the surgery mapping cone formula by Ozsv\'ath and Szab\'o. Let $C=\CFKi(Y,K)$ be the reduced filtered knot Floer homology for a knot $K$ in integer homology sphere $Y$. The filtered mapping cone formula puts extra filtrations on the mapping cone formula for $\HFp(Y_1(K))$, which in turns gives the knot Floer complex $\CFKi(Y_1(K),\K)$, where $\K$ is the core of the surgery solid torus of the $+1$-surgery on $K$.

Let $C=\bigoplus_{i,j\in \Z} C(i,j)$.
For each integer $s$, let $A_s^\infty$ and $B_s^\infty$ each denote a copy of the chain complex $C$, and write $A_s^\infty = \bigoplus_{i,j \in \Z} A_s^\infty(i,j)$ and $B_s^\infty = \bigoplus_{i,j \in \Z} B_s^\infty(i,j)$. We define a pair of $\Z$-filtrations $\II$ and $\JJ$ on each of these complexes as follows:
\begin{align*}
\II(A_s^\infty(i,j)) &= \max(i,j-s)  &  \JJ(A_s^\infty(i,j)) &= \max(i+s-1, j) \\
\II(B_s^\infty(i,j)) &= i & \JJ(B_s^\infty(i,j)) &= i+s-1.
\end{align*}
Let $A_s^-$ (resp. $B_s^-$) denote the subcomplex of $A_s^\infty$ (resp. $B_s^\infty$) with $\II<0$, let $A_s^+$ (resp. $B_s^+$) denote the quotient, and let $\hat A_s$ (resp. $\hat B_s$) be the subcomplex of $A_s^+$ (resp. $B_s^+$) with $\II=0$. This definition of $A_s^\circ$ and $B_s^\circ$ coincides with the definition by Ozsv\'ath and Szab\'o. 

Next we define the maps $v_s^\circ \co A_s^\circ \to B_s^\circ$ and $h_s^\circ \co A_s^\circ \to B_{s+1}^\circ$. Let $v_s^\infty$ be the identity map of $C$, and let $h_s^\infty$ be $\phi$ composed with multiplication by $U^s$, where $\phi$ is a $U$-equivariant, grading-preserving chain homotopy equivalence $\phi \co C \to C$, which restricts to a homotopy equivalence between the subcomplexes $C\{j \le s\}$ and $C \{i \le s\}$. Each of these maps is filtered with respect to both $\II$ and $\JJ$. In particular, $v_s^\infty$ (resp. $h_s^\infty$) takes the subcomplex $A_s^-$ into $B_s^-$ (resp. $B_{s+1}^-$), and induces a map $v_s^+ \co A_s^+ \to B_s^+$ (resp. $h_s^+ \co A_s^+ \to B_{s+1}^+$), which agrees with the definition in \cite{OS04}. Moreover, each of $v_s^\infty$ and $h_s^\infty$ is homogeneous of degree $-1$ with respect to the Maslov grading.
Let
\[
\Psi^\infty_{1-g,g} \co \bigoplus_{s=1-g}^g A_s^\infty \to \bigoplus_{s=2-g}^g B_s^\infty
\]
be the map given by the sum of the maps $v_s^\infty \co A_s^\infty \to B_s^\infty$ ($s=2-g, \dots, g$) and $h_s^\infty \co A_s^\infty \to B_{s+1}^\infty$ ($s = 1-g, \dots, g-1$), and let $\X^\infty$ denote the mapping cone of $\Psi^\infty_{1-g,g}$. Filtration functions $\II$ and $\JJ$ give $\X^\infty$ the structure of a doubly filtered chain complex with an action of $\F[U,U^{-1}]$. 

The following is implied by the main result from \cite{HeddenLevine}:

\begin{theorem} [Theorem 1.1 in \cite{HeddenLevine}]\label{thm:surgery}
The chain complex $\X^\infty$ is filtered quasi-isomorphic to $\CFKi(Y_1(K), \tilde K)$, where the filtrations $\II$ and $\JJ$ on $\X^\infty$ correspond to $i$ and $j$ on $\CFKi(Y_1(K), \tilde K)$.
\end{theorem}

For a filtered basis of $\X^\infty$,  we have $\d^\infty = \d + \d'$, where $\d$ consists of the terms that preserve both $\II$ and $\JJ$ filtrations and $\d'$ consists of the terms that lower at least one of the filtrations.

In order to compute invariants such as $\tau$ and $\varep$ more effectively, one wants to obtain a \emph{reduced} complex for $\CFKi(Y_1(K), \tilde K)$, namely a complex where each term in the differential strictly lowers at least one of the filtrations. There is a ``cancellation'' mathod (See Proposition 11.57 in \cite{Bordered}) that allows us to perform a filtered change of basis to obtain a reduced complex of $\X^\infty$. We will describe its procedure as follows. In each summand $\X(0,s)$, one can pick a basis $\{y_i\}$ for $\Im(\d)$, and then select elements $x_i \in \X(0,s)$ such that $\d(x_i) = y_i$. Then $\partial^\infty(x_i) = y_i + {}$ terms in lower filtration levels. The subcomplex of $\X^\infty$ spanned (over $\F[U,U^{-1}]$) by all the $\{x_i, \d^\infty(x_i)\}$ is acyclic, and quotienting $\X^\infty$ by this subcomplex yields a reduced complex for $\CFKi(Y_1(K), \tilde K)$. The generators for this quotient complex (over $\F[U,U^{-1}]$) are in one-to-one correspondence with the generators (over $\F$) of $\HFKh(Y_1(K), \tilde K)$, and the differential is induced from the terms in $\d'$ which strictly lower one of the filtrations. From the above analysis, we see that it is practical to start from each individual complex $\X(0,s)$ to compute $\HFKh(Y_1(K), \tilde K)$. We identify the generators in a reduced complex for $\CFKi(Y_1(K), \tilde K)$ with the generators of $\HFKh(Y_1(K), \tilde K)$ and compute  the induced differentials.
Next, we compute an example utilizing this  method.

\subsection*{An example: $\CFKi{(S^{3}_{+1}(T_{3,5}),\tilde{T}_{3,5})}$}

As an example, we compute  $\CFKi{(S^{3}_{+1}(T_{3,5}),\tilde{T}_{3,5})}$ using the filtered mapping cone formula, where $T_{3,5}$ is the $(3,5)$-torus knot, and $\tilde{T}_{3,5}$ is the core circle of the surgery solid torus in $1$-surgery on $T_{3,5}$.  See \cite{HomLevineLidman} Section $6$ for a more complicated example.
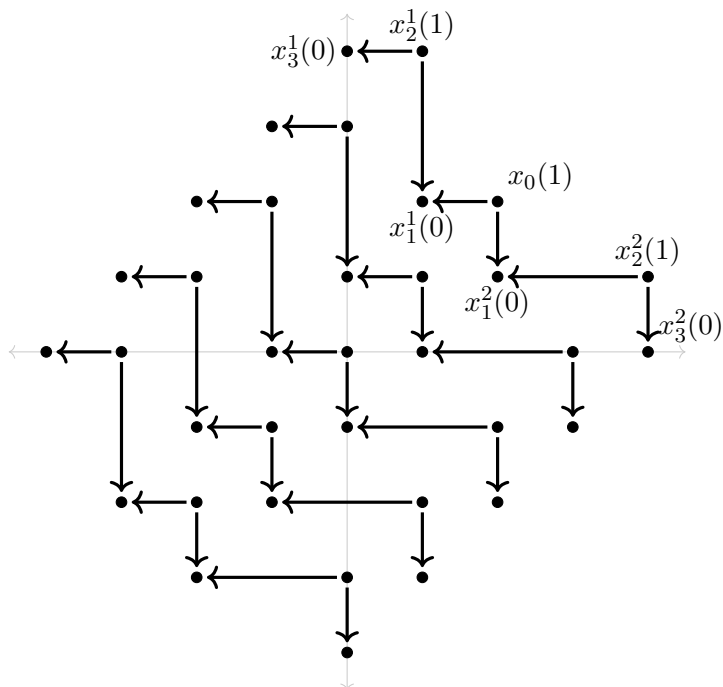
\begin{figure}[htb!]
\begin{tikzpicture}
	\begin{scope}[thin, black!20!white]
		\draw [<->] (-4, 0.5) -- (5, 0.5);
		\draw [<->] (0.5, -4) -- (0.5, 5);
	\end{scope}

\foreach \x in {-3,...,0}
{	
	
	\filldraw (\x-0.5, \x+3.5) circle (2pt) node[] (a){};
	\filldraw (\x+0.5, \x+3.5) circle (2pt) node[] (b){};
	\filldraw (\x+0.5, \x+1.5) circle (2pt) node[] (c){};
	\filldraw (\x+1.5, \x+1.5) circle (2pt) node[] (d){};
	\filldraw (\x+1.5, \x+0.5) circle (2pt) node[] (e){};
	\filldraw (\x+3.5, \x+0.5) circle (2pt) node[] (f){};
	\filldraw (\x+3.5, \x-0.5) circle (2pt) node[] (g){};

	\draw [very thick, ->] (b) -- (c);
	\draw [very thick, ->] (b) -- (a);
	\draw [very thick, ->] (d) -- (e);
	\draw [very thick, ->] (d) -- (c);
	\draw [very thick, ->] (f) -- (g);
	\draw [very thick, ->] (f) -- (e);

}	
\foreach \x in {1}
{	
	
	\filldraw (\x-0.5, \x+3.5) circle (2pt) node[] (a){};
	\filldraw (\x+0.5, \x+3.5) circle (2pt) node[] (b){};
	\filldraw (\x+0.5, \x+1.5) circle (2pt) node[] (c){};
	\filldraw (\x+1.5, \x+1.5) circle (2pt) node[] (d){};
	\filldraw (\x+1.5, \x+0.5) circle (2pt) node[] (e){};
	\filldraw (\x+3.5, \x+0.5) circle (2pt) node[] (f){};
	\filldraw (\x+3.5, \x-0.5) circle (2pt) node[] (g){};

	\draw [very thick, ->] (b) -- (c);
	\draw [very thick, ->] (b) -- (a);
	\draw [very thick, ->] (d) -- (e);
	\draw [very thick, ->] (d) -- (c);
	\draw [very thick, ->] (f) -- (g);
	\draw [very thick, ->] (f) -- (e);

	\node [left] at (a) {$x^{1}_{3}(0)$};
	\node [above] at (b) {$x^{1}_{2}(1)$};
	\node [below] at (c) {$x^{1}_{1}(0)$};
	\node [above right] at (d) {$x_{0}(1)$};
	\node [below] at (e) {$x^{2}_{1}(0)$};
	\node [above] at (f) {$x^{2}_{2}(1)$};
 \node [above right] at (g) {$x^{2}_{3}(0)$};
}

\end{tikzpicture}
\caption{The complex $\CFKi(T_{3,5})$. Numbers in parenthesis indicate Maslov gradings.}
\label{fig:CFKiT35}
\end{figure}

 Starting from $\CFKi{(T_{3,5})}$, here we use the same notations for generators as in \cite{HendricksManolescu}. Since $g(T_{3,5})=4$, the complex $\X^\infty$ is a mapping cone
\[
\bigoplus_{s=-3}^4 A_s^\infty \to \bigoplus_{s=-2}^4 B_s^\infty,
\]
We will work out $\CFKh{(S^{3}_{+1}(T_{3,5}),\tilde{T}_{3,5},1)}$ explicitly; the remaining computations are similar and left to the
reader. Reader can find the information of all generators in Table \ref{tab:CFKiKreduce}.
To start, recall $\CFKh{(S^{3}_{+1}(T_{3,5}),\tilde{T}_{3,5},1)}$ is given by the mapping cone
\begin{equation} 
 A_1\{i \leqslant 0, j=1\} \oplus A_2\{i = 0, j \leqslant 1\} \xrightarrow {(h_1, \ v_2)} B_2\{i=0\}.
\end{equation}
We depict $A_1$, $A_2$, and $B_2$ in Figures \ref{subfig:A1}, \ref{subfig:A2}, and \ref{subfig:B2}
respectively. The generators of different $A_i$ and $B_i$ are distinguished by the indice outside the parentheses. The generators of $A_i$ are labeled with an $i$-subscript outside the parentheses and the generators of $B_i$
are labeled with a prime and an $i$-subscript outside the parentheses. 

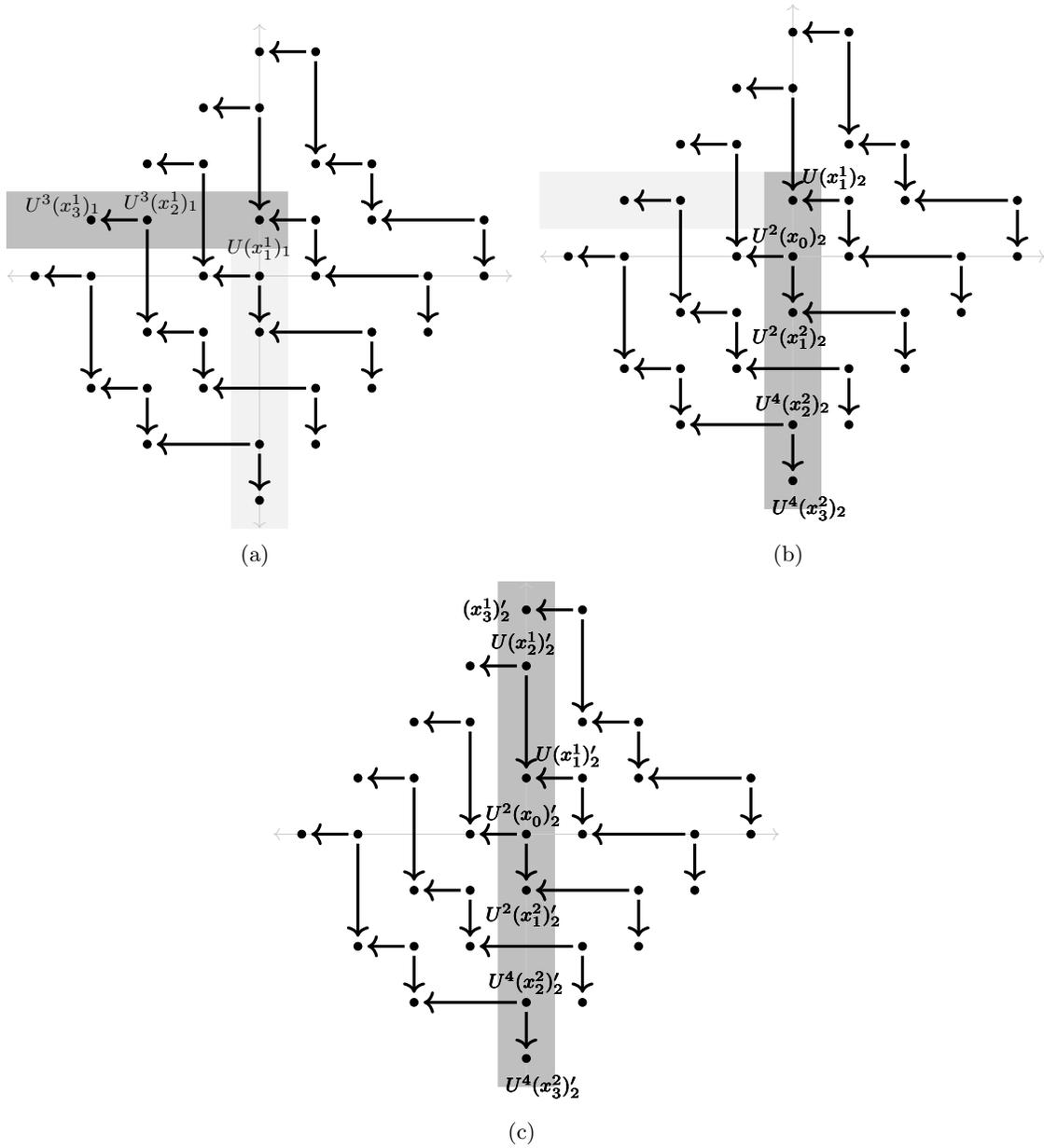
\begin{figure}[htb!]
\subfigure[]{
\begin{tikzpicture}[scale=0.8]

\filldraw[black!5!white] (0, -4) rectangle (1, 1);
\filldraw[black!25!white] (-4, 1) rectangle (1, 2);

	\begin{scope}[thin, black!20!white]
		\draw [<->] (-4, 0.5) -- (5, 0.5);
		\draw [<->] (0.5, -4) -- (0.5, 5);
	\end{scope}

\foreach \x in {-3,...,1}
{	
	
	\filldraw (\x-0.5, \x+3.5) circle (2pt) node[] (a){};
	\filldraw (\x+0.5, \x+3.5) circle (2pt) node[] (b){};
	\filldraw (\x+0.5, \x+1.5) circle (2pt) node[] (c){};
	\filldraw (\x+1.5, \x+1.5) circle (2pt) node[] (d){};
	\filldraw (\x+1.5, \x+0.5) circle (2pt) node[] (e){};
	\filldraw (\x+3.5, \x+0.5) circle (2pt) node[] (f){};
	\filldraw (\x+3.5, \x-0.5) circle (2pt) node[] (g){};

	\draw [very thick, ->] (b) -- (c);
	\draw [very thick, ->] (b) -- (a);
	\draw [very thick, ->] (d) -- (e);
	\draw [very thick, ->] (d) -- (c);
	\draw [very thick, ->] (f) -- (g);
	\draw [very thick, ->] (f) -- (e);

}	
\node[] at (-3,1.75) {\scriptsize$U^3 (x^{1}_{3})_{1}$};
\node[] at (-1.25,1.8) {\scriptsize$U^3 (x^{1}_{2})_{1}$};
\node[] at (0.5,1) {\scriptsize$U (x^{1}_{1})_{1}$};

\end{tikzpicture}
\label{subfig:A1}
}
\subfigure[]{
\begin{tikzpicture}[scale=0.8]

\filldraw[black!5!white] (-4, 1) rectangle (0, 2);
\filldraw[black!25!white] (0, -4) rectangle (1, 2);

	\begin{scope}[thin, black!20!white]
		\draw [<->] (-4, 0.5) -- (5, 0.5);
		\draw [<->] (0.5, -4) -- (0.5, 5);
	\end{scope}

\foreach \x in {-3,...,1}
{	
	
	\filldraw (\x-0.5, \x+3.5) circle (2pt) node[] (a){};
	\filldraw (\x+0.5, \x+3.5) circle (2pt) node[] (b){};
	\filldraw (\x+0.5, \x+1.5) circle (2pt) node[] (c){};
	\filldraw (\x+1.5, \x+1.5) circle (2pt) node[] (d){};
	\filldraw (\x+1.5, \x+0.5) circle (2pt) node[] (e){};
	\filldraw (\x+3.5, \x+0.5) circle (2pt) node[] (f){};
	\filldraw (\x+3.5, \x-0.5) circle (2pt) node[] (g){};

	\draw [very thick, ->] (b) -- (c);
	\draw [very thick, ->] (b) -- (a);
	\draw [very thick, ->] (d) -- (e);
	\draw [very thick, ->] (d) -- (c);
	\draw [very thick, ->] (f) -- (g);
	\draw [very thick, ->] (f) -- (e);

\node[] at (1.25,1.9) {\scriptsize$U (x^{1}_{1})_{2}$};
\node[] at (0.45,0.85) {\scriptsize$U^2 (x^{}_{0})_{2}$};
\node[] at (0.45,-0.95) {\scriptsize$U^2 (x^{2}_{1})_{2}$};
\node[] at (0.5,-2.15) {\scriptsize$U^4 (x^{2}_{2})_{2}$};
\node[] at (0.8,-4) {\scriptsize$U^4 (x^{2}_{3})_{2}$};

}	
\end{tikzpicture}
\label{subfig:A2}
}
\subfigure[]{
\begin{tikzpicture}[scale=0.8]

\filldraw[black!25!white] (0, -4) rectangle (1, 5);

	\begin{scope}[thin, black!20!white]
		\draw [<->] (-4, 0.5) -- (5, 0.5);
		\draw [<->] (0.5, -4) -- (0.5, 5);
	\end{scope}

\foreach \x in {-3,...,1}
{	
	
	\filldraw (\x-0.5, \x+3.5) circle (2pt) node[] (a){};
	\filldraw (\x+0.5, \x+3.5) circle (2pt) node[] (b){};
	\filldraw (\x+0.5, \x+1.5) circle (2pt) node[] (c){};
	\filldraw (\x+1.5, \x+1.5) circle (2pt) node[] (d){};
	\filldraw (\x+1.5, \x+0.5) circle (2pt) node[] (e){};
	\filldraw (\x+3.5, \x+0.5) circle (2pt) node[] (f){};
	\filldraw (\x+3.5, \x-0.5) circle (2pt) node[] (g){};

	\draw [very thick, ->] (b) -- (c);
	\draw [very thick, ->] (b) -- (a);
	\draw [very thick, ->] (d) -- (e);
	\draw [very thick, ->] (d) -- (c);
	\draw [very thick, ->] (f) -- (g);
	\draw [very thick, ->] (f) -- (e);

\node[] at (-0.2,4.5) {\scriptsize$(x^{1}_{3})'_{2}$};
\node[] at (0.45,3.9) {\scriptsize$U (x^{1}_{2})'_{2}$};
\node[] at (1.25,1.9) {\scriptsize$U (x^{1}_{1})'_{2}$};
\node[] at (0.45,0.85) {\scriptsize$U^2 (x^{}_{0})'_{2}$};
\node[] at (0.45,-0.95) {\scriptsize$U^2 (x^{2}_{1})'_{2}$};
\node[] at (0.5,-2.15) {\scriptsize$U^4 (x^{2}_{2})'_{2}$};
\node[] at (0.8,-4) {\scriptsize$U^4 (x^{2}_{3})'_{2}$};

}	
\end{tikzpicture}
\label{subfig:B2}
}
\caption{Top left, $A_1$. Top right, $A_2$. Bottom, $B_2$. The dark grey regions comprise
 $\CFKh{(S^{3}_{+1}(T_{3,5}),\tilde{T}_{3,5},1)}$.}
\label{fig:CFK_1core}
\end{figure}

From Figure \ref{fig:CFK_1core} one can read off that $\CFKh{(S^{3}_{+1}(T_{3,5}),\tilde{T}_{3,5},1)}$ consists $4$ direct summands under the basis we choose. We summarize them in Figure \ref{fig:4sum}.

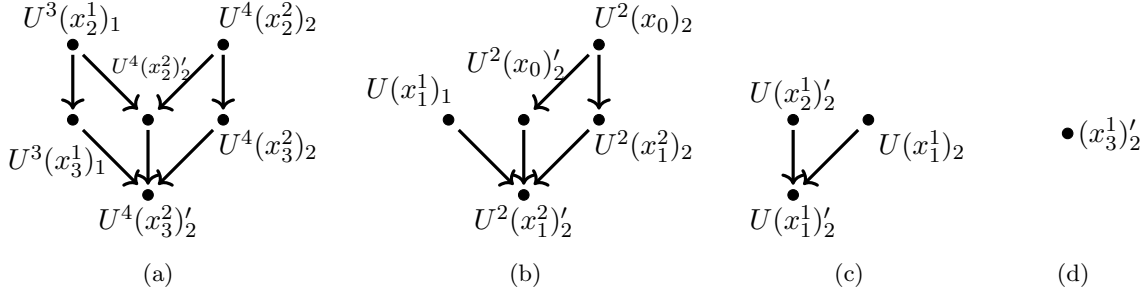
\begin{figure}
\subfigure[]{
\begin{tikzpicture}

\filldraw (-0.5,1.5) circle (2pt) node[] (a){};
\filldraw (-0.5,0.5) circle (2pt) node[] (b){};
\filldraw (0.5,0.5) circle (2pt) node[] (c){};
\filldraw (0.5,-0.5) circle (2pt) node[] (d){};
\filldraw (1.5,1.5) circle (2pt) node[] (e){};
\filldraw (1.5,0.5) circle (2pt) node[] (f){};

    \draw [very thick, ->] (a) -- (b);
	\draw [very thick, ->] (a) -- (c);
	\draw [very thick, ->] (c) -- (d);
	\draw [very thick, ->] (e) -- (f);
	\draw [very thick, ->] (e) -- (c);
	\draw [very thick, ->] (b) -- (d);
	\draw [very thick, ->] (f) -- (d);
	
\node [above] at (a) {$U^{3}(x^1_2)_1$};
\node [below] at (-0.7,0.3) {$U^{3}(x^1_3)_1$};
\node [above] at (0.55,0.9) {\scriptsize$U^{4}(x^2_2)'_2$};
\node [below] at (d) {$U^{4}(x^2_3)'_2$};
\node [above right] at (1.3,1.5) {$U^{4}(x^2_2)_2$};
\node [below right] at (1.3,0.5) {$U^{4}(x^2_3)_2$};

\end{tikzpicture}
}
\subfigure[]{
\begin{tikzpicture}

\filldraw (-0.5,0.5) circle (2pt) node[] (b){};
\filldraw (0.5,0.5) circle (2pt) node[] (c){};
\filldraw (0.5,-0.5) circle (2pt) node[] (d){};
\filldraw (1.5,1.5) circle (2pt) node[] (e){};
\filldraw (1.5,0.5) circle (2pt) node[] (f){};

	\draw [very thick, ->] (c) -- (d);
	\draw [very thick, ->] (e) -- (f);
	\draw [very thick, ->] (e) -- (c);
	\draw [very thick, ->] (b) -- (d);
	\draw [very thick, ->] (f) -- (d);
	
\node [above left] at (-0.3,0.55) {$U(x^1_1)_1$};

\node [above] at (0.4,0.9) {$U^{2}(x_0)'_2$};
\node [below] at (d) {$U^{2}(x^2_1)'_2$};
\node [above right] at (1.3,1.5) {$U^{2}(x_0)_2$};
\node [below right] at (1.3,0.5) {$U^{2}(x^2_1)_2$};
\end{tikzpicture}
}
\subfigure[]{
\begin{tikzpicture}

	\begin{scope}[thin, black!0!white]
		\draw [<->] (-0.3, 0.5) -- (0.8, 0.5);

	\end{scope}
	
\filldraw (0.5,0.5) circle (2pt) node[] (c){};
\filldraw (0.5,-0.5) circle (2pt) node[] (d){};

\filldraw (1.5,0.5) circle (2pt) node[] (f){};

	\draw [very thick, ->] (c) -- (d);
	
	\draw [very thick, ->] (f) -- (d);

\node [above] at (c) {$U(x^1_2)'_2$};
\node [below] at (d) {$U(x^1_1)'_2$};

\node [below right] at (f) {$U(x^1_1)_2$};
\end{tikzpicture}
}
\subfigure[]{
\begin{tikzpicture}
\begin{scope}[thin, black!0!white]
		\draw [<->] (-0.3, 0.5) -- (0.8, 0.5);
        \draw [<->] (-0.3, -0.5) -- (0.8, 1.5);
	\end{scope}
\filldraw (0.5,1) circle (2pt) node[] (c){};

\node [right] at (c) {$(x^1_3)'_2$};

\end{tikzpicture}
}
\caption{The four direct summands in $\CFKh{(S^{3}_{+1}(T_{3,5}),\tilde{T}_{3,5},1)}$}
\label{fig:4sum}
\end{figure}

The first summand is acyclic whereas the other three summands each generates one dimensional homology.

It follows that  $\HFKh{(S^{3}_{+1}(T_{3,5}),\tilde{T}_{3,5},1)}$ is generated by
\[
\{     U (x^{1}_{1})_{1}+U^2 (x_0)'_2, \quad U (x^{1}_{1})_{2}+U (x^1_2)'_2,   \quad   (x^1_3)'_2  \}
\]

The next step is to pick a reduced basis for $\CFKi{(S^{3}_{+1}(T_{3,5}),\tilde{T}_{3,5})}$. We summarize all generators and their $\CFKi$ differentials in Table \ref{tab:CFKiKreduce}. The notations we choose here, even though cumbersome, will be helpful in the more general case. Note that the isolated generator $g_{j-1}:=(x^1_3)'_{j}$ in $B_{j}$ always survives as a generator of Alexander grading $j-1$ for $j\in \{-2,...,4\}$; we choose pairs of image and preimage under map $\d$ for  $j\in \{-2,...,4\}$:
\[
\d U (x^1_2)'_j=U (x^1_1)'_j, \quad   \d U^2 (x_0)'_j=U^2 (x^2_1)'_j, \quad  \d U^4 (x^2_2)'_j=U^4 (x^2_3)'_j
\]
and compute
\begin{align*}
\d^\infty  (x^1_2)'_j&= (x^1_3)'_j  +(x^1_1)'_j\\
   \d^\infty  (x_0)'_j&= (x^1_1)'_j + (x^2_1)'_j\\
     \d^\infty  (x^2_2)'_j&= (x^2_1)'_j + (x^2_3)'_j
\end{align*}
After quotienting by the $\F[U,U^{-1}]$-submodule $S$ generated by 
\[
 \{
  (x^1_2)'_j, \quad (x_0)'_j, \quad (x^2_2)'_j, \quad, (x^1_3)'_j  +(x^1_1)'_j, \quad (x^1_1)'_j + (x^2_1)'_j, \quad   (x^2_1)'_j + (x^2_3)'_j
 \}
\]
it readily follows that $(x^1_3)'_j  =(x^1_1)'_j = (x^2_1)'_j = (x^2_3)'_j = g_{j-1}$.

\begin{table}[htb!]
\begin{center}
\begin{tabular}{*{6}{cccccl}}
\hline
Alexander gr. & & Generator &  Maslov gr. & $\d^\infty$ & \\ \hline
$4$ & $\beta_4$   & $(x^1_3)_4$ & $12$ & $(x^1_3)'_4$ & $=g_3$ \\
$3$ & $g_3 $   &  $(x^1_3)'_4$ & $11$ & $0$ & \\
$2$ & $g_2$   & $(x^1_3)'_3$ & $5$ & $0$ & \\
$2$ & $\alpha_2$   &  $U (x^1_1)_3 + U (x^1_2)'_3$ & $4$ & $U(x^1_3)'_3 +U^4 (x^2_1)'_4$ & $=U g_2+U^4 g_3$\\
$1$ & $g_1$   &  $(x^1_3)'_2$ & $1$ & $0$ & \\
$1$ & $\beta_1$   &  $U (x^1_1)_1 + U^2 (x_{0})'_2 $ & $-2$ & $U (x^1_1)'_1 + U^2 (x^1_1)'_2$ & $=U g_0  +U^2 g_1$ \\
$1$ & $\alpha_1$   &  $U (x^1_1)_2 +U (x^1_2)'_2$ & $0$ & $U (x^1_3)'_2 + U^3 (x^2_1)'_3$ & $=U g_1 +U^3 g_2$ \\
$0$ & $g_0$   &  $(x^1_3)'_1$ & $-1$ & $0$ & \\
$-1$ & $g_{-1}$   &  $(x^1_3)'_0$ & $-1$ & $0$ \\
$-1$ & $\beta_{-1}$   &  $U^2 (x^2_1)_{-1} + U (x^1_2)'_0$ & $-2$ & $U (x^1_3)'_0 + U^2 (x^2_1)'_{-1}$  & $=U g_{-1} +U^2 g_{-2}$ \\
$-1$ & $\alpha_{-1}$   & $U^2 (x^2_1)_0 + U^2 (x_0)'_0$ & $-4$ & $U^2 (x^1_1)'_0 +U^2 (x^1_1)'_1$  & $=U^2 g_{-1} + U^3 g_{-3}$ \\
$-2$ & $g_{-2}$   &  $(x^1_3)'_{-1}$ & $1$ & $0$ & \\
$-2$ & $\beta_{-2}$   &  $U^3 (x^2_1)_{-2} + U (x^1_2)'_{-1}$ & $0$ & $U (x^1_3)'_{-1} + U^3 (x^2_1)'_{-2}$ & $=Ug_{-2} +U^3 g_{-3}$ \\
$-3$ & $g_{-3}$   &  $(x^1_{3})'_{-2}$ & $5$ & $0$ & \\
$-4$ & $\alpha_{-4}$   &  $U^4 (x^2_3)_{-3}$ & $4$ & $U (x^1_3)'_{-2}$ & $=U g_{-3}$
\\ \hline
\end{tabular}
\end{center}
\caption{Summary of the generators of $\X^\infty$ which survive in the reduced complex for  $\CFKi{(S^{3}_{+1}(T_{3,5}),\tilde{T}_{3,5})}$. 
}
\label{tab:CFKiKreduce}
\end{table}

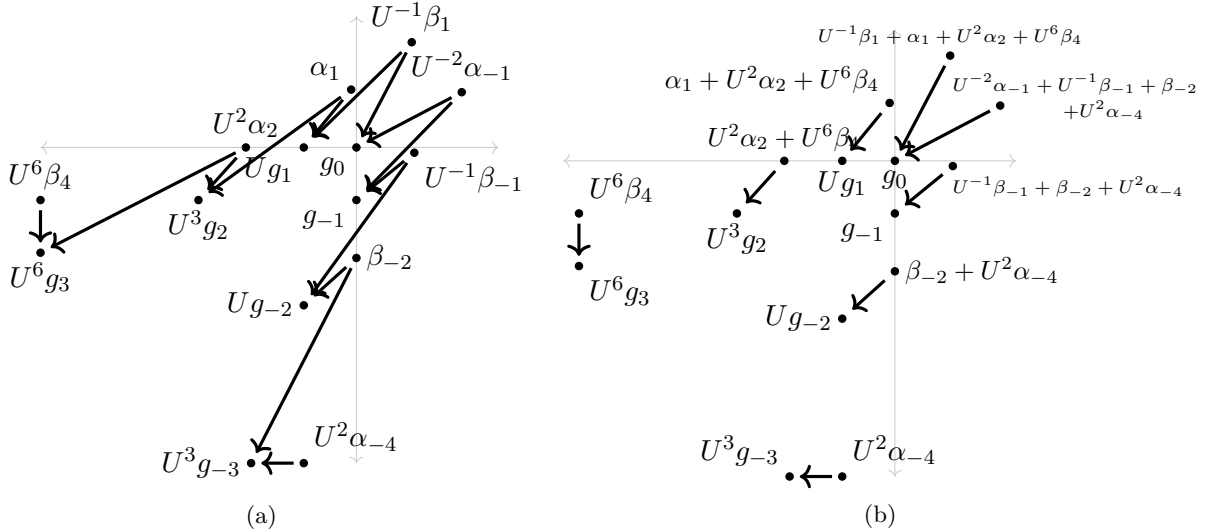
\begin{figure}
\subfigure[]{
\begin{tikzpicture}[scale=0.7]

	\begin{scope}[thin, black!20!white]
		\draw [<->] (-5.5, 0.5) -- (3.2, 0.5);
		\draw [<->] (0.5, -5.5) -- (0.5, 3);
	\end{scope}

	\filldraw (-5.5,-0.5) circle (2pt) node[] (a){};
	\filldraw (-5.5,-1.5)  circle (2pt) node[] (b){};
	\filldraw (-1.6,0.5)  circle (2pt) node[] (c){};
	\filldraw (-2.5,-0.5)  circle (2pt) node[] (d){};
	\filldraw (0.4,1.6)  circle (2pt) node[] (e){};
	\filldraw (-0.5,0.5)  circle (2pt) node[] (f){};
	\filldraw (1.55,2.5)  circle (2pt) node[] (g){};
	\filldraw (0.5,0.5)  circle (2pt) node[] (h){};	
	\filldraw (2.5,1.55)  circle (2pt) node[] (i){};	
	\filldraw (0.5,-0.5)  circle (2pt) node[] (j){};
	\filldraw (1.6,0.4)  circle (2pt) node[] (k){};
	\filldraw (-0.5,-2.5)  circle (2pt) node[] (l){};
	\filldraw (0.5,-1.6)  circle (2pt) node[] (m){};
	\filldraw (-1.5,-5.5)  circle (2pt) node[] (n){};
	\filldraw (-0.5,-5.5)  circle (2pt) node[] (o){};
					
	\draw [very thick, ->] (a) -- (b);
	\draw [very thick, ->] (c) -- (b);
	\draw [very thick, ->] (c) -- (d);
	\draw [very thick, ->] (e) -- (d);
	\draw [very thick, ->] (e) -- (f);
	\draw [very thick, ->] (g) -- (f);
	\draw [very thick, ->] (g) -- (h);
	\draw [very thick, ->] (i) -- (h);
	\draw [very thick, ->] (i) -- (j);
	\draw [very thick, ->] (k) -- (j);
	\draw [very thick, ->] (k) -- (l);
	\draw [very thick, ->] (m) -- (l);
	\draw [very thick, ->] (m) -- (n);
	\draw [very thick, ->] (o) -- (n);
	
	\node[above] at (a) {$U^6 \beta_4$};
	\node[below] at (b) {$U^6 g_3$};
	\node[above] at (c) {$U^2 \alpha_2$};
	\node[below] at (d) {$U^3 g_2$};
	\node[above left] at (e) {$\alpha_1$};
	\node[below left] at (f) {$U g_1$};
	\node[above] at (g) {$U^{-1} \beta_1$};
  \node[below left] at (h) {$g_0$};
	\node[above] at (i) {$U^{-2}\alpha_{-1}$};
	\node[below left] at (j) {$g_{-1}$};
	\node[below right] at (k) {$U^{-1}\beta_{-1}$};
	\node[left] at (l) {$U g_{-2}$};
	\node[right] at (m) {$\beta_{-2}$};
	\node[left] at (n) {$U^3 g_{-3}$};
	\node[above right] at (o) {$U^2 \alpha_{-4}$};

\end{tikzpicture}

\label{fig:Kreduce}
}
\subfigure[]{
\begin{tikzpicture}[scale=0.7]

	\begin{scope}[thin, black!20!white]
		\draw [<->] (-5.8, 0.5) -- (2.8, 0.5);
		\draw [<->] (0.5, -5.5) -- (0.5, 3);
	\end{scope}

	\filldraw (-5.5,-0.5) circle (2pt) node[] (a){};
	\filldraw (-5.5,-1.5)  circle (2pt) node[] (b){};
	\filldraw (-1.6,0.5)  circle (2pt) node[] (c){};
	\filldraw (-2.5,-0.5)  circle (2pt) node[] (d){};
	\filldraw (0.4,1.6)  circle (2pt) node[] (e){};
	\filldraw (-0.5,0.5)  circle (2pt) node[] (f){};
	\filldraw (1.55,2.5)  circle (2pt) node[] (g){};
	\filldraw (0.5,0.5)  circle (2pt) node[] (h){};	
	\filldraw (2.5,1.55)  circle (2pt) node[] (i){};	
	\filldraw (0.5,-0.5)  circle (2pt) node[] (j){};
	\filldraw (1.6,0.4)  circle (2pt) node[] (k){};
	\filldraw (-0.5,-2.5)  circle (2pt) node[] (l){};
	\filldraw (0.5,-1.6)  circle (2pt) node[] (m){};
	\filldraw (-1.5,-5.5)  circle (2pt) node[] (n){};
	\filldraw (-0.5,-5.5)  circle (2pt) node[] (o){};
					
	\draw [very thick, ->] (a) -- (b);
	
	\draw [very thick, ->] (c) -- (d);

	\draw [very thick, ->] (e) -- (f);

	\draw [very thick, ->] (g) -- (h);
	\draw [very thick, ->] (i) -- (h);

	\draw [very thick, ->] (k) -- (j);
	
	\draw [very thick, ->] (m) -- (l);

	\draw [very thick, ->] (o) -- (n);
	
	\node[above right] at (a) {$U^6 \beta_4$};
	\node[below right] at (b) {$U^6 g_3$};
	\node[above] at (c) {\small$U^2 \alpha_2 +U^6 \beta_4 $};
	\node[below] at (d) {$U^3 g_2$};
	\node[above left] at (e) {\small$\alpha_1 + U^2 \alpha_2 +U^6 \beta_4 $};
	\node[below] at (f) {$U g_1$};
	\node[above] at (g) {\tiny$U^{-1} \beta_1 + \alpha_1 + U^2 \alpha_2 +U^6 \beta_4$};
  \node[below] at (h) {$g_0$};

		\node[above right] at (1.4,1.55) {\tiny$U^{-2}\alpha_{-1}+U^{-1}\beta_{-1}+\beta_{-2}$};
		\node[above right] at (3.5,1.05) {\tiny$+U^2 \alpha_{-4}$};
	\node[below left] at (j) {$g_{-1}$};
	\node[below right] at (1.4,0.4) {\tiny$U^{-1}\beta_{-1}+\beta_{-2}+U^2 \alpha_{-4}$};
	\node[left] at (l) {$U g_{-2}$};
	\node[right] at (m) {\small$\beta_{-2}+U^2 \alpha_{-4}$};
	\node[above left] at (n) {$U^3 g_{-3}$};
	\node[above right] at (o) {$U^2 \alpha_{-4}$};

\end{tikzpicture}

\label{fig:Kchangebasis}
}
\caption{The reduced complex for $\CFKi{(S^{3}_{+1}(T_{3,5}),\tilde{T}_{3,5})}$, drawn in the $(i,j)$-plane. The right figure is after the change of basis.}
\end{figure}

We perform a change of basis that yields Figure \ref{fig:Kchangebasis}.  On the top half of the graph, we replace $U^2 \alpha_2$ by $U^2 \alpha_2 +U^6 \beta_4 $, replace $\alpha_1$ by $\alpha_1 + U^2 \alpha_2 +U^6 \beta_4 $ and $U^{-1}\beta_{-1}$ by $U^{-1}\beta_{-1} + \alpha_1 + U^2 \alpha_2 +U^6 \beta_4 $. We change the basis for the bottom half of the graph in a similar way, and keep the rest generators unchanged. Quotienting all the acyclic summands from Figure \ref{fig:Kchangebasis} leaves us with a complex with $3$ generators from which one can readily read out that $\tau(S^{3}_{+1}(T_{3,5}),\tilde{T}_{3,5})=-1$ and $\varep(S^{3}_{+1}(T_{3,5}),\tilde{T}_{3,5})=0$.

\section{Computation for general L-Space knots}

In this section we prove Theorem \ref{thm:Lspace}.
First, recall from Section 3 that if $K$ is an L-space knot, then the Alexander polynomial of $K$ is of the form
$$ \Delta_K(t) = (-1)^m + \sum_{i=1}^m (-1)^{m-i}(t^{n_i} + t^{-n_i})$$
for a sequence of positive integers $0 < n_1 < n_2 < \dots < n_m.$ Here, $n_m = g(K)$ is the genus of $K$. Let $n(K) \geqslant 0$ be the quantity
$$ n(K):= n_m - n_{m-1} + \dots + (-1)^{m-2} n_2 + (-1)^{m-1} n_1.$$ Furthermore, let $\ell_s = n_s - n_{s-1}$.

We compute a reduced complex for $\CFKi(S^3_1(K),\K)$ according to the method we described in Section 4 (See the paragraph after Theorem \ref{thm:surgery}). Recall that we identify the generators in the reduced complex of $\CFKi(S^3_1(K),\K)$ with the generators in $\HFKh(S^3_1(K),\K)$ and compute the induced differential where each term lowers at least one of the $\II$ and $\JJ$ filtrations. Throughout this section, when we say a generator in the reduced complex of $\CFKi(S^3_1(K),\K)$ has Alexander filtration level $j$, without specification it means that in $\HFKh(S^3_1(K),\K)$ the generator has Alexander grading $\JJ=j$.

Each generator in $\CFKi(S^3_1(K),\K)$  has a coordinate in $\Z\oplus \Z$ given by $(\II,\JJ)$. Define $\prec$ to be the natural partial order in $\Z\oplus \Z$, where $(a,b) \prec (c,d)$ if and only if $a\leqslant c $ and $b\leqslant d $.

Theorem \ref{thm:Lspace} follows immediately from the following observations. Indices are assigned according to the same rule as in Section \ref{sec:T35}: the generators of $A_j$ are labeled with an $j$-subscript outside the parentheses and the generators of $B_j$
are labeled with a prime and an $j$-subscript outside the parentheses. 

\begin{proposition} \label{pro:lspace}
Given an L-space knot $K$ with Alexander polynomial 
$$ \Delta_K(t) = (-1)^m + \sum_{i=1}^m (-1)^{m-i}(t^{n_i} + t^{-n_i}),$$
we can choose a basis such that in the reduced complex $\CFKi(S^3_1(K),\K)$ the following are satisfied.
\begin{enumerate}[label=(\roman{*}), ref=\roman{*}]
	\item \label{it:gj} $g_j := (x^1_m)'_{j+1}$ is a generator in Alexander filtration level $j$ for $-n_m+1\leqslant j\leqslant n_m-1$. We call all $g_j$ {\em lower corner generators}, and the rest generators {\em upper corner generators}. We abuse notation and also call their images under $\F [U,U^{-1}]$-actions lower or upper corner generators, respectively.
	\item \label{it:connect} There are at most two upper corner generators in filtration level $j$, which, if they exist, we denote by $\alpha_j$ and $\beta_j$, where
	\begin{align*}
	\d^\infty \alpha_j &= U^{a_1}g_j + U^{a_2}g_{j+1}, \qquad -n_m+1\leqslant j\leqslant n_m-1 \\
		\d^\infty \beta_j &= U^{b_1}g_j + U^{b_2}g_{j-1}, \qquad -n_m+1\leqslant j\leqslant n_m-1 \\
	&	\d^\infty \beta_{n_m}   \ \ = g_{n_m -1}, \\
	&  \d^\infty \alpha_{-n_m}  = Ug_{-n_m +1}.
	\end{align*}
	where $a_i,b_i,i\in\{1,2\}$ are some constants (also depend on $j$). Each $g_j$ is in the image of exactly two upper corner generators.
	\item \label{it:comparable} Given two upper corner generators $u_j, u_{j'}$ in $j,j'$ filtration level respectively with $j>j'\geqslant 0$, suppose for some constants $t$ and $s$, $U^t u_j$ and $U^s u_{j'}$ have the same Maslov grading and their coordinates are $(i_1,j_1),(i_2,j_2)$ respectively, then $(i_1,j_1)\prec (i_2,j_2)$. 
	\item \label{it:comparable2} Suppose $\alpha_j, \beta_j$ both exist in filtration level $j \geqslant 0$. If for some constants $t$ and $s$, $U^t \alpha_j$ and $U^s \beta_{j}$ have the same Maslov grading and their coordinates are $(i_1,j_1),(i_2,j_2)$ respectively, then $(i_1,j_1)\prec (i_2,j_2)$.
\end{enumerate}
\end{proposition}

Before proving this result, we show how it implies Theorem \ref{thm:Lspace}.
\begin{proof}[Proof of Theorem \ref{thm:Lspace}]
Starting from the top half of the complex, $\beta_{n_m} :=(x^1_m)_{n_m}$ is the only generator in Alexander level $n_m$. According to \eqref{it:connect},  $\d^{\infty}\beta_{n_m} = g_{n_m -1}$ and since $g_{n_m -1}$ is in the image of exactly two upper corner generators, we have $\d^{\infty} ( \text{another upper corner generator} )  = g_{n_m -1} + U^{c}g_{n_m -2},$ where $c$ is some constant. Similarly $g_{n_m -2}$ is in the image of exactly two upper corner generators, so we can find the next upper corner generator that has $g_{n_m -2}$ in the differential. Notice that the next upper corner generator we find in this way will have non increasing $j$ filtration level due to the relation in \eqref{it:connect}. Repeating the process until we reach filtration level $0$, this determines a sequence of all the upper corner generators with filtration level $j\geqslant 0$ where adjacent elements in the sequence share a same lower corner generator in the differential.

 We then perform a change of basis similar to the one in Figure \ref{fig:Kchangebasis}: for each $\beta_j$ where $j\geqslant 0$, replace it with the sum of itself, all the $\alpha_j'$ with $j' \geqslant j$ and all the $\beta_j'$ with $j' > j$, multiplying each term by $U$ or $U^{-1}$ if necessary such that they have the same Maslov grading; for each $\alpha_j$ where $j\geqslant 0$, replace it with the sum of itself and all the $\alpha_j'$ and $\beta_j'$ with $j' > j$, multiplying each term by $U$ or $U^{-1}$ if necessary such that they have the same Maslov grading. By \eqref{it:comparable} and \eqref{it:comparable2} this is a filtered change of basis. The symmetry of knot complex allows us to change the basis for the bottom half of the complex in a similar way. Quotienting out all acyclic summands reduces $\CFKi(S^3_1(K),\K)$ to a complex with $3$ generators. A direct computation shows that for $m$ odd case, the quotient complex consists of $\beta_1, g_0, \alpha_{-1}$ and is determined by the relations (after the quotient) :
\[
 \d^\infty \alpha_{-1} =U^{n(K)}g_0 , \quad \d^\infty \beta_1 =U^{n(K)-1}g_0 
\]
For the $m$ even case, the quotient complex consists of $\beta_0, g_0, \alpha_0$ with relations
\[
\d^\infty \beta_0 =U^{n(k)}g_{0} , \quad \d^\infty \alpha_0 =U^{n(k)}g_{0} 
\]
and a further change of basis yields a complex with one generator.
\end{proof}
The rest of the section is dedicated to prove the Proposition \ref{pro:lspace}.
In order to prove \eqref{it:gj}, simply notice $(x^1_m)'_j$ is always in the kernel of $\d$, and never in the image of $v_j,h_{j-1}$.

Before proceeding to the proof of \eqref{it:connect}, we choose a reduced basis that will simplify the computation. We demonstrate the case when $m$ is odd here, while the even case is analogous. Recall that $\ell_s=n_s-n_{s-1}$, and let  $L_s:=\ell_m+\ell_{m-2}+...+\ell_{s+2}$ where $m-s$ is a positive even integer.
Choose pairs of image and preimage of map $\d$ for $j\in \{ -n_m +2,.., n_m \}$:
\begin{align*}
 \d U^{L_{2i-1}} (x^1_{2i})'_j&=U^{L_{2i-1}} (x^1_{2i-1})'_j \quad 1\leqslant i\leqslant \frac{m-1}{2} \\
 \d U^{n(K)} (x_0)'_j& =U^{n(K)} (x^2_1)'_j \\
 \d U^{n(K)+\sum^i_1{\ell_{2k}}} (x^2_{2i})'_j& =U^{n(K)+\sum^i_1{\ell_{2k}}} (x^2_{2i+1})'_j \quad 1\leqslant i\leqslant \frac{m-1}{2}
\end{align*}
and compute
\begin{align*}
 \d^\infty (x^1_{2i})'_j&=  (x^1_{2i+1})'_j + (x^1_{2i-1})'_j \quad 1\leqslant i\leqslant \frac{m-1}{2} \\
 \d^\infty (x_0)'_j& = (x^1_1)'_j+  (x^2_1)'_j \\
\d^\infty  (x^2_{2i})'_j& = (x^2_{2i+1})'_j + (x^2_{2i-1})'_j \quad 1\leqslant i\leqslant \frac{m-1}{2}
\end{align*}
After quotienting by the $\F[U,U^{-1}]$-submodule $S$ generated by 
\[
 \{
  (x^1_{2i})'_j, \quad \d^\infty (x^1_{2i})'_j, \quad  (x_0)'_j, \quad, \d^\infty (x_0)'_j, \quad (x^2_{2i})'_j, \quad   \d^\infty  (x^2_{2i})'_j
 \} \quad 1\leqslant i\leqslant \frac{m-1}{2}
\]
it readily follows that $(x^1_{2i+1})'_j  =(x^2_{2i+1})'_j = g_{j-1}$ for $0\leqslant i\leqslant \frac{m-1}{2}$. Similarly, in the case where $m$ is even we get $(x^{1}_{2i})'_j =(x_0)'_j=(x^{2}_{2i})'_j =g_{j-1} $ for $ 1\leqslant i\leqslant \frac{m}{2}, \quad j\in \{  -n_m+2,...,n_m \}$.

 To show \eqref{it:connect}, we observe there are $4$ different cases depending on the choice of Alexander filtration $j$ of the generators in $\CFKh(S^3_1(K),\K)$, see Figure \ref{fig:ab} . 
 Recall that $\CFKh(S^3_1(K),\K, n_m)$ is given by $A_{n_m}\{i' \leqslant 0, j'=n_m\}$, and  for $j<n_m,$ $\CFKh(S^3_1(K),\K, j)$ is given by the mapping cone (Here $i'$ and $j'$ are used to denote the $i$ and $j$ filtration in $\CFKi(S^3,K)$, respectively. )
 \begin{equation*} 
 A_j\{i' \leqslant 0, j'=j\} \oplus A_{j+1}\{i' = 0, j' \leqslant j\} \xrightarrow {(h_j, \ v_{j+1})} B_{j+1}\{i'=0\}.
\end{equation*}

 For the following computation, we only care about the case when $j\geqslant 0$ and adopt the convention $x^1_0 :=x_0=: x^2_0$. 
\begin{enumerate}[label=(\alph{*}), ref=\alph{*}]
\item \label{it:ab1} Figure \ref{subfig:case1} represents the case when $n_{s-1}<j<n_s$ where $m-s$ is a positive even integer, note that the mapping cone consists of two summands that generate homology: the isolated generator
 $g_j :=(x^1_m)'_{j+1}$ in $B_{j+1}$, and a summand with $5$ generators consists of the right most generator in $A_j\{i' \leqslant 0, j'=j\}, U^{L_{s}+(n_{s}-j)} (x^1_s)_j$, $2$ generators in $B_{j+1}, U^{L_{s}+n_{s}}(x^2_{s-1})'_{j+1}$ and $U^{L_{s}+n_{s}}(x^2_{s})'_{j+1}$, and $2$ generators in $A_{j+1}, U^{L_{s}+n_{s}}(x^2_{s-1})_{j+1}$ and $U^{L_{s}+n_{s}}(x^2_{s})_{j+1}$.
In this case $U^{L_{s}+(n_{s}-j)} (x^1_s)_j+U^{L_{s}+n_{s}}(x^2_{s-1})'_{j+1}$ is the only upper corner generator , which consists of elements from $A_{j}$ and $B_{j+1}$. We call this kind of generator a {\em $\beta$ type generator}, and label them $\beta_j$. Direct computation yields $\d^\infty \beta_j =U^{L_{s}+(n_s-j)} g_{j-1} + U^{L_{s}+n_s} g_j$. 
\item \label{it:ab2} Figure \ref{subfig:case2} represents the case when $j=n_{s-1}$ where $m-s=2i,i\in \N$. The isolated generator $g_{n_{s-1}}$ is the only generator in this case.
\item \label{it:ab3} Figure \ref{subfig:case3} represents the case when $n_s<j<n_{s+1}$ where $m-s=2i,i>0$. Aside from $g_j$, the only generator is $U^{L_{s}} (x^1_s)_{j+1} + U^{L_{s}}(x^1_{s+1})'_{j+1}$, which consists of elements from $A_{j+1}$ and $B_{j+1}$. We call this kind of generator an {\em $\alpha$ type generator}, and label them $\alpha_j$. $\d^\infty \alpha_j = U^{L_{s} + j +1} g_{j+1} + U^{L_s} g_j$.
\item \label{it:ab4} Figure \ref{subfig:case4} represents the case when $j=n_s$ where $m-s=2i,i>0$. We have in this case three generators $\beta_{n_s}= U^{L_s} (x^1_s)_{n_s} +U^{L_s +n_s} (x^2_{s-1})'_{n_s +1}$, $\alpha_{n_s} := U^{L_s} (x^1_s)_{n_s +1} + U^{L_s} (x^1_{s+1})'_{n_s+1}$ and $g_{n_s}$. We compute $\d^\infty \beta_{n_s} = U^{L_s} g_{n_s -1} + U^{L_{s}+n_s} g_{n_s}$,$\d^\infty \alpha_{n_s}= U^{L_s +n_s +1} g_{n_s +1} + U^{L_s} g_{n_s}$.
\end{enumerate}
The above computation along with $\d^\infty \beta_{n_m} = g_{n_m -1}$ gives explicit relations for all generators on the top half.  After multiplying the upper corner generators by $U$ and $U^{-1}$ if necessary to make them have the same Maslov grading, we analyse the four cases for $j$ as follows
\begin{itemize}
\item[--]  case \eqref{it:ab1}, $g_j$ is in  $\d^{\infty}U^{a_1} \beta_{j+1}$ and  in $\d^{\infty}U^{a_2} \beta_{j}$, for some $a_1,a_2\in \Z$;
\item[--]  case \eqref{it:ab2}, $g_j$ is in  $\d^{\infty}U^{b_1} \beta_{j+1}$ and  in $\d^{\infty}U^{b_2} \alpha_{j-1}$, for some $b_1,b_2\in \Z$;
\item[--]  case \eqref{it:ab3}, $g_j$ is in  $\d^{\infty}U^{c_1} \alpha_j$       and  in $\d^{\infty}U^{c_2} \alpha_{j-1}$, for some $c_1,c_2\in \Z$;
\item[--]  case \eqref{it:ab4}, $g_j$ is in  $\d^{\infty}U^{d_1}\alpha_{j} $    and  in $\d^{\infty}U^{d_2} \beta_j$, for some $d_1,d_2\in \Z$;
   \end{itemize}
   This completes the proof of \eqref{it:connect}.
\begin{figure}
\subfigure[]{
\begin{tikzpicture}[scale=1]

	\begin{scope}[thin, black!20!white]
		\draw [<->] (0.5, -2) -- (0.5, 3);
	\end{scope}

\filldraw (0.5,0.5) circle (2pt) node(a){};

\draw [very thick] (0.5,-0.5) -- (-0.2,-0.5);
\draw [very thick] (0.5,-0.5) -- (0.5,-2);
\draw [very thick] (0.5,3) -- (0.5,1.5)	;
\draw [very thick] (0.5,1.5) -- (2.5,1.5)	;
\draw [very thick] (2.5,1.5) -- (2.5,-1.5)	;
\draw [very thick] (2.5,-1.5) -- (3,-1.5)	;
\draw [thick, dotted] (0.5,0.5) -- (-0.5,0.5);

\node [] at (-0.4,0.8) {$j$};
\node [] at (1.05,1.75) {\small $U^{L_s}x^1_s$};
\node [] at (2.8,1.8) {\small $U^{L_s}x^1_{s-1}$};
\node [] at (1.75,-1.6) {\small $U^{L_s}x^1_{s-2}$};
\node [above right] at (0.5,-0.5) {\small $U^{L_{s-2}}x^1_{s-1}$};
\node [below left] at (0.5,-0.5) {$n_{s-1}$};
\node [left] at (0.5,1.5) {$n_{s}$};

\end{tikzpicture}
\label{subfig:case1}
}
\subfigure[]{
\begin{tikzpicture}[scale=1]

	\begin{scope}[thin, black!20!white]
		\draw [<->] (0.5, -2) -- (0.5, 3);
	\end{scope}

\filldraw (0.5,0.5) circle (2pt) node(a){};

\draw [very thick] (0.5,0.5) -- (-0.2,0.5);
\draw [very thick] (0.5,0.5) -- (0.5,-1.75);
\draw [very thick] (0.5,3) -- (0.5,2.5)	;
\draw [very thick] (0.5,2.5) -- (2.5,2.5)	;
\draw [very thick] (2.5,2.5) -- (2.5,-0.5)	;
\draw [very thick] (2.5,-0.5) -- (3,-0.5)	;
\draw [thick, dotted] (0.5,0.5) -- (-0.5,0.5);

\node [] at (-0.4,0.8) {$j$};
\node [] at (1.05,2.75) {\small $U^{L_s}x^1_s$};
\node [] at (2.8,2.8) {\small $U^{L_s}x^1_{s-1}$};
\node [] at (1.75,-0.6) {\small $U^{L_s}x^1_{s-2}$};
\node [above right] at (0.5,0.5) {\small $U^{L_{s-2}}x^1_{s-1}$};
\node [below left] at (0.5,0.5) {$n_{s-1}$};
\node [left] at (0.5,2.5) {$n_{s}$};

\end{tikzpicture}
\label{subfig:case2}
}
\subfigure[]{
\begin{tikzpicture}[scale=1]

	\begin{scope}[thin, black!20!white]
		\draw [<->] (0.5, -2) -- (0.5, 3);
	\end{scope}

\filldraw (0.5,1) circle (2pt) node(a){};

\draw [very thick] (0.5,2.5) -- (0.5,-1);
\draw [very thick] (0.5,-1) -- (1,-1);
\draw [very thick] (0.5,2.5) -- (-0.2,2.5)	;

\draw [thick, dotted] (0.5,1) -- (-0.5,1);

\node [right] at (0.5,2.5) {$U^{L_{s} x^1_{s+1}}$}; 
\node [above right] at (0.5,-1) {$U^{L_{s} x^1_{s}}$}; 
\node [] at (-0.4,1.2) {$j$};
\node [below left] at (0.5,2.5) {$n_{s+1}$};
\node [below left] at (0.5,-1) {$n_{s}$};
\end{tikzpicture}
\label{subfig:case3}
}
\subfigure[]{
\begin{tikzpicture}[scale=0.8]

	\begin{scope}[thin, black!20!white]
		\draw [<->] (0.5, -2) -- (0.5, 3);
	\end{scope}

\filldraw (0.5,-1) circle (2pt) node(a){};

\draw [very thick] (0.5,2.5) -- (0.5,-1);
\draw [very thick] (0.5,-1) -- (1,-1);
\draw [very thick] (0.5,2.5) -- (-0.2,2.5)	;

\draw [thick, dotted] (0.5,-1) -- (-0.5,-1);

\node [right] at (0.5,2.5) {$U^{L_{s} x^1_{s+1}}$}; 
\node [above right] at (0.5,-1) {$U^{L_{s} x^1_{s}}$}; 
\node [] at (-0.4,-0.8) {$j$};
\node [below left] at (0.5,2.5) {$n_{s+1}$};
\node [below left] at (0.5,-1) {$n_{s}$};
\end{tikzpicture}
\label{subfig:case4}
}
\label{fig:ab}
\caption{Four different situations depending on $j$ filtration of $\CFKh(S^3_1(K),\K, j)$.   The dotted line shows $A_j\{i' \leqslant 0, j'=j\}$.}
\end{figure}
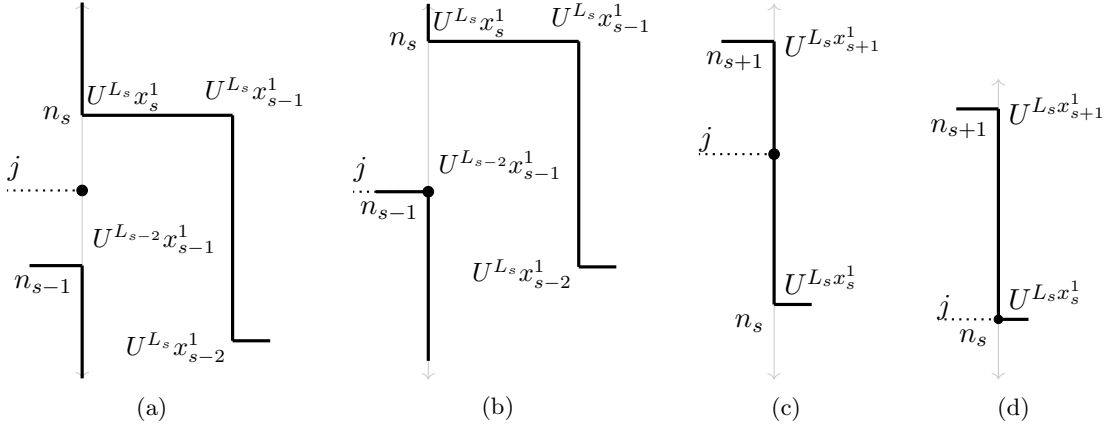

We prove \eqref{it:comparable} and \eqref{it:comparable2} together. From  the final part of the discussion in the proof of \eqref{it:connect} it is clear there are $4$ cases where adjacent upper corner generators share a same lower corner generator in the differential, see Figure \ref{fig:upperconnect}.

The proof is to verify in each of the $4$ cases the conclusion is correct. We will carry out the computation in case \eqref{it:ab1}; see Figure \ref{subfig:case1} and Figure \ref{subfig:up1}. The other three cases are similar and left for the reader.

In this case $j$ satisfies $n_{s-1}<j<n_s$ where $m-s=2i,i\in \N$, 
\begin{align*}
\beta_j & =U^{L_{s}+( n_s -j )} (x^1_s)_j + U^{L_s +n_s}(x^2_{s-1})'_{j+1}, \\
\d^\infty \beta_j & =U^{L_{s}+( n_s -j )} g_{j-1} + U^{L_s +n_s}g_{j},  \\
\beta_{j+1} & =U^{L_{s}+( n_s -j -1 )} (x^1_s)_{j+1} + U^{L_s +n_s}(x^2_{s-1})'_{j+2}, \\
\d^\infty \beta_{j+1} & =U^{L_{s}+( n_s -j -1 )} g_{j} + U^{L_s +n_s}g_{j+1}.
\end{align*}
Thus $U^{-(L_s +n_s -j -1) } \beta_{j+1}$ and $U^{-(L_s + n_s)} \beta_{j}$ have the same Maslov grading. Clearly we have $(L_s+n_s -j -1, L_s + n_s)\prec (L_s + n_s, L_s + n_s + j)$.

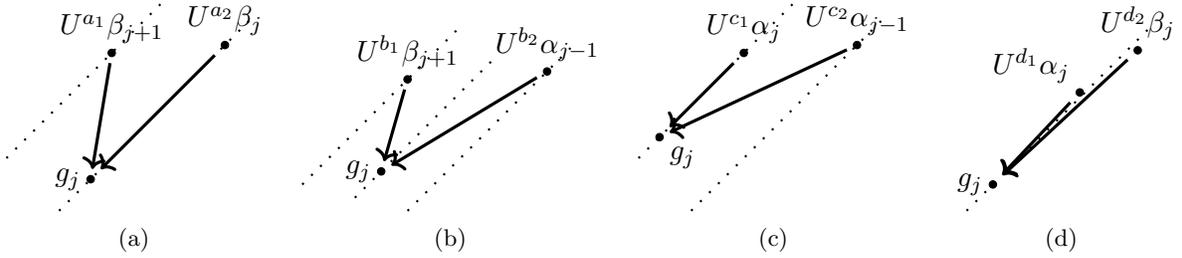
\begin{figure}
\subfigure[]{
\begin{tikzpicture}[scale=0.7]

	\draw [thick, loosely dotted] (-2.5,-0.5)--(0.5,2.5);
  \draw [thick, loosely dotted] (-1.5,-1.5)--(2.2,2.2);
	
	\filldraw (-0.5,1.5) circle (2pt) node (a) {};
	\filldraw (1.65,1.65) circle (2pt) node (b) {};
	\filldraw (-0.9,-0.9) circle (2pt) node (c) {};
	\draw [very thick, ->] (a)--(c);
	\draw [very thick, ->] (b)--(c);
		
		\node [above] at (a) {$U^{a_1} \beta_{j+1}$};
		\node [above] at (b) {$U^{a_2} \beta_{j}$};
		\node [left] at (c) {$g_{j}$};
	
	\end{tikzpicture}
	\label{subfig:up1}
	}
	\subfigure[]{
\begin{tikzpicture}[scale=0.7]
	
	\draw [thick, loosely dotted] (-2.5,-0.5)--(0.5,2.5);
  \draw [thick, loosely dotted] (-1.5,-0.75)--(1.2,1.95);
  \draw [thick, loosely dotted] (-0.5,-1)--(2.7,2.2);
	
	\filldraw (-0.5,1.5) circle (2pt) node (a) {};
	\filldraw (2.15,1.65) circle (2pt) node (b) {};
	\filldraw (-1,-0.25) circle (2pt) node (c) {};
	\draw [very thick, ->] (a)--(c);
	\draw [very thick, ->] (b)--(c);
		
		\node [above] at (a) {$U^{b_1} \beta_{j+1}$};
		\node [above] at (b) {$U^{b_2} \alpha_{j-1}$};
		\node [left] at (c) {$g_{j}$};

	\end{tikzpicture}
	\label{subfig:up2}
	}
	\subfigure[]{
\begin{tikzpicture}[scale=0.7]

	\draw [thick, loosely dotted] (-2.5,-0.5)--(0.5,2.5);
  \draw [thick, loosely dotted] (-1.5,-1.5)--(2.2,2.2);
	
	\filldraw (-0.5,1.5) circle (2pt) node (a) {};
	\filldraw (1.65,1.65) circle (2pt) node (b) {};
	\filldraw (-2.1,-0.1) circle (2pt) node (c) {};
	\draw [very thick, ->] (a)--(c);
	\draw [very thick, ->] (b)--(c);
		
		\node [above] at (a) {$U^{c_1} \alpha_{j}$};
		\node [above] at (b) {$U^{c_2} \alpha_{j-1}$};
		\node [below right] at (c) {$g_{j}$};
	
	\end{tikzpicture}
	\label{subfig:up3}
	}
	\subfigure[]{
\begin{tikzpicture}[scale=0.7]

	 \draw [thick, loosely dotted] (-1.5,-0.75)--(2,2.75);
	 
	 \filldraw (1.75,2.3) circle (2pt) node (a) {};
	\filldraw (0.65,1.5) circle (2pt) node (b) {};
	\filldraw (-1,-0.25) circle (2pt) node (c) {};
	\draw [very thick, ->] (a)--(c);
	\draw [very thick, ->] (b)--(c);
		
				\node [above] at (a) {$U^{d_2} \beta_{j}$};
		\node [above left] at (b) {$U^{d_1} \alpha_{j}$};

		\node [left] at (c) {$g_{j}$};
	 
	\end{tikzpicture}
	\label{subfig:up4}
	}
	\caption{The four possible situations where adjacent upper corner generators share a same lower corner generator in the differential. Each dotted line consists of points with coordinates $(i,j)$ where $i-j = $ constant. }
	\label{fig:upperconnect}
	\end{figure}

\section{Connected knot floer complex}

In this section, we review the local equivalence studied in \cite{ZemConn} and define the connected knot Floer complex, using the same argument in the construction of the connected Heegaard Floer homology in \cite{HHL}.

\subsection{Knot Floer homology}
 In the previous parts of the paper, we look at the knot complex in the ring $\F[U]$, but now it is necessary to move to a refined version as we want the extra grading information it encodes. 
Let $\F[\U,\V]$ be the polynomial ring with $2$ variables over finite field $\F=\Z/2\Z$. Here we denote the new variables $\U,\V$ to distinguish from the variable of $\F[U]$.
We start from reviewing some background of knot Floer homology over the ring  $\F[\U,\V]$.

 For a knot $K$ in a integer homology sphere $Y$, let $\cH = (\Sigma, \bfalpha, \bfbeta, w, z)$ be a doubly-pointed Heegaard diagram compatible with $(Y, K)$. Define $\CFKUV(K)$ to be the chain complex freely generated over the ring $\mathcal{R}=\F[\U,\V]$ by $\bfx \in \bbT_{\bfalpha} \cap \bbT_{\bfbeta}$ with differential
\[ \d \bfx = \sum_{\bfy \in \bbT_{\bfalpha} \cap \bbT_{\bfbeta}}  \sum_{\substack{\phi \in \pi_2(\bfx, \bfy) \\ \mu(\phi)=1}} \U^{n_w(\phi)} \V^{n_z(\phi)} \bfy, \]
where, as usual, $\pi_2(\bfx, \bfy)$ denotes Whitney disks connecting $\bfx$ to $\bfy$, and $\mu(\phi)$ denotes the Maslov index of $\phi$.
  
Define the relative bigradings $\gr_U$  and  $\gr_V$ as follow:
\begin{align*}
	\gr_U(\bfx, \bfy) &= \mu(\phi) -2n_w(\phi) \\
	\gr_V(\bfx, \bfy) &= \mu(\phi) -2n_z(\phi). 
\end{align*}
 where  $\gr_U$ is the previous Maslov grading and $\gr_V$ is the grading after exchanging the role of two base points in the Heegaard Floer data. We often refer to $\gr_U$ as $U$-grading and  $\gr_V$ $V$-grading. Define the relative {\em Alexander grading}  
 by
$$ A(\bfx, \bfy) = \frac{1}{2} (\gr_U(\bfx, \bfy) - \gr_V(\bfx, \bfy)) = n_z(\phi) - n_w(\phi). $$

\begin{figure}
\subfigure[]{
\label{subfig:A0}
\begin{tikzpicture}

\begin{scope}[thin, black!20!white]
\end{scope}

\foreach \x in {-2,...,0}
{	
	
	\filldraw (\x-0.5, \x+0.5) circle (2pt) node[] (a){};
	\filldraw (\x+0.5, \x+0.5) circle (2pt) node[] (b){};
	\filldraw (\x+0.5, \x-0.5) circle (2pt) node[] (c){};

	\draw [very thick, ->] (b) -- (a);
	\draw [very thick, ->] (b) -- (c);

}	
\node[above] at (0.5,0.5) {\scriptsize$x$};
\node[above left] at (-0.5,0.5) {\scriptsize$\U y$};
\node[above right] at (0.5,-0.5) {\scriptsize$\V z$};

\node[above] at (-0.5,-0.5) {\scriptsize$\U\V x$};
\node[above left] at (-1.5,-0.5) {\scriptsize$\U^2\V y$};
\node[above right] at (-0.5,-1.5) {\scriptsize$\U\V^2 z$};

\node[above] at (-1.5,-1.5) {\scriptsize$\U^2\V^2 x$};
\node[above left] at (-2.5,-1.5) {\scriptsize$\U^3\V^2 y$};
\node[above right] at (-1.5,-2.5) {\scriptsize$\U^2\V^3 z$};
\end{tikzpicture}
}
\subfigure[]{
\label{subfig:A1}
\begin{tikzpicture}

\begin{scope}[thin, black!20!white]
\end{scope}

\foreach \x in {-2,...,0}
{	
	
		\filldraw (\x+0.5, \x+1.5) circle (2pt) node[] (a){};
	\filldraw (\x+0.5, \x+0.5) circle (2pt) node[] (b){};
	\filldraw (\x+0.5, \x-0.5) circle (2pt) node[] (c){};

	\draw [very thick, ->] (b) -- (c);
        }
     \filldraw (-2.5,-1.5) circle(2pt) node [] (d){};
	
	\draw [very thick, ->] (-1.5,-1.5) -- (d);

   \draw [very thick, ->] (0.42,0.5) -- (-0.42,0.5);

	\draw [very thick, ->] (-0.42,-0.5) -- (-1.42,-0.5);
	
\node[above] at (0.5,0.5) {\scriptsize$\V x$};
\node[above left] at (0.5,1.5) {\scriptsize$ y$};
\node[above right] at (0.5,-0.5) {\scriptsize$\V^2 z$};

\node[above] at (-0.5,-0.5) {\scriptsize$\U\V^2 x$};
\node[above left] at (-0.5,0.5) {\scriptsize$\U\V y$};
\node[above right] at (-0.5,-1.5) {\scriptsize$\U\V^3 z$};

\node[above] at (-1.5,-1.5) {\scriptsize$\U^2\V^3 x$};
\node[above left] at (-1.5,-0.5) {\scriptsize$\U^2\V^2 y$};
\node[above right] at (-1.5,-2.5) {\scriptsize$\U^2\V^4 z$};

\node[above left] at (-2.5,-1.5) {\scriptsize$\U^3\V^3 y$};
\end{tikzpicture}
}
\label{fig:T23}
\caption{ Knot complex $\CFKUV(T_{2,3})$, where $T_{2,3}$ is the right handed trefoil. Figure \ref{subfig:A0} shows the Alexander grading  $0$ complex $\CFKUV (K, 0)$, while Figure \ref{subfig:A1} shows the Alexander graing $1$ complex $\CFKUV (K, 1)$. }
\end{figure}
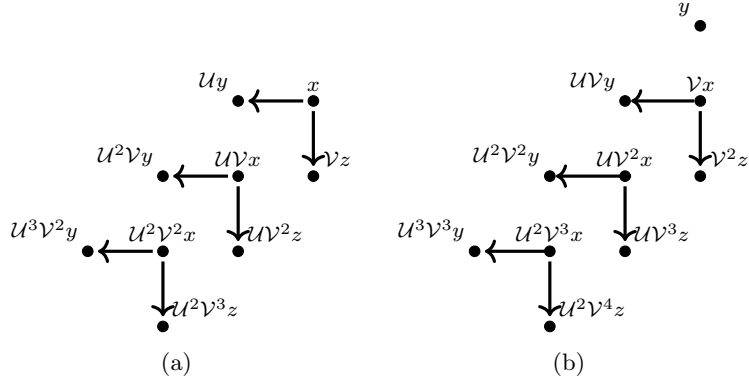

Note that $\U$ lowers $\gr_U$ by $2$, preserves $\gr_V$, and lowers $A$ by $1$, while $\V$ preserves $\gr_U$, lowers $\gr_V$ by $2$, and increases $A$ by $1$. Importantly, the differential preserves  Alexander grading. Hence the complex $\CFKUV(K)$ splits as an $\F[\U,\V]$-module over the Alexander grading:
\[ \CFKUV (K) = \bigoplus_{s \in \Z} \CFKUV (K, s), \]
The chain complex $\CFKUV (K, s)$ is isomorphic to the complex $A^{-}_s$ (See the beginning of Section 4 for the definition). One can think of  $\CFKUV(K)$ as a direct sum of $A^{-}_s$.  We have  
\begin{align*}
\U &\co  \CFKUV (K, s) \to  \CFKUV (K, s-1) \\
\V &\co  \CFKUV (K, s) \to  \CFKUV (K, s+1).
\end{align*}

We identify $\CFKi(K)$ with $\CFKUV(K,0) \otimes \F[\U,\V,(\U\V)^{-1}]$, where $\F[\U,\V,(\U\V)^{-1}]$ is the localization of $\F[\U,\V]$ at the ideal $(\U\V)$. While previously multiplication by $U$ is the chain isomorphism that shift the complex down the diagonal by $1$ unit, now this action is refined to be multiplication by $\U\V$. The complex $\CFKUV(K)$ is generated by elements without $\U,\V$ decoration.
\subsection{Self-local equivalence}

Local equivalence encodes the information of homology concordance. This fact allows us to extend various invariants originally defined only for knots in $S^{3}$ to homology concordance invariants for knots in arbitrary integer homology spheres. To say it more precisely, we have the following:

\begin{definition}[Definition 2.4 in \cite{ZemConn}]
For knots $K_1,K_2$ in integer homology spheres $Y_1,Y_2$ respectively, if there exist  $\F[\U,\V]$-equivariant chain maps 
\[ f \co \CFKUV(Y_1, K_1) \rightarrow \CFKUV(Y_2, K_2) \qquad \textup{ and } \qquad g \co \CFKUV(Y_2, K_2) \rightarrow \CFKUV(Y_1, K_1) \]
such that $f$ and $g$ preserve absolute $\U,\V$ gradings and induce isomorphisms on $(\U\V)^{-1}H_{*}$, then $\CFKUV(Y_1,K_1)$ and $\CFKUV(Y_2,K_2)$ are {\em locally equivalent}.
\end{definition}

Here  for a complex $C$,  $(\U\V)^{-1}H_*(C)$  is defined to be $H_*(C\otimes_{\F[\U,\V]}  \F[\U,\V,(\U\V)^{-1}])$, where $\F[\U,\V,(\U\V)^{-1}]$ is the localization of $\F[\U,\V]$ at the ideal $(\U\V)$.  For any knot $K$ in an integer homology sphere $Y$, there is a relatively bi-graded isomorphism between
$  (\U\V)^{-1}H_{*}(\CFKUV(Y,K))$ and  $\F[\U,\V,(\U\V)^{-1}]. $ We will denote local equivalence by $\simeq$.  It is easy to verify that $\simeq$ is an equivalence relation.

\begin{proposition}[Theorem 1.5 in \cite{ZemConn}]\label{prop:zem}
If $(Y_1,K_1)$ and $(Y_2,K_2)$ are homology concordant, then
 $$\CFKUV(Y_1,K_1) \simeq \CFKUV(Y_2,K_2).$$
\end{proposition}

\begin{definition}
Let $\J$ be the set of local equivalence class of chain complexes over the ring $\F[\U,\V]$. The operation over $\J$ is induced by tensor product of chain complexes. For complex $C\in \J$, define the inverse  of  $C$  to be the dual complex $C^*$   and the identity element is $\F[\U,\V]$.
\end{definition}

\begin{proposition}[Proposition 2.6 in \cite{ZemConn}]
 $\J$ is a well-defined abelian group.
\end{proposition}

Taking into the consideration of the above proposition, we will use additive notion to indicate the tensor product of the chain complexes. 

Using the relation of local equivalence from a complex to itself, we can define the connected knot complex, which will be our key ingredient in the proof. We start by recalling the definition of self-local equivalence.

\begin{definition}[Definition 3.1 in \cite{HHL}]
Let $C=\CFKUV(Y,K)$, where $K$ is a knot in an integer homology sphere $Y$. An absolute bigrading preserving chain map
$$  f\co C \longrightarrow C  $$
is a {\em self-local equivalence} if $f$ induces an isomorphism on $(\U\V)^{-1}H_{*}$.
\end{definition}

The set of self-local equivalences come with a preorder defined by comparing the kernel. A self-local equivalence $f$ is called {\em maximal} if for any other self-local equivalence $g$, ker $f$ $\subset $ ker $g$ implies ker $f$ $=$ ker $g$. Maximal self-local equivalences always exist since $C$ is finitely generated. 

\begin{lemma}[Lemma 3.4 in \cite{HHL}] \label{le:max}
If $f, g \co  C \rightarrow C $ are maximal self-local equivalences of $C $, then  $f|_{\mathrm{im} g} \co \mathrm{im} g \rightarrow  \mathrm{im} f$
 is an isomorphism of chain complexes.
\end{lemma}
\begin{proof}
Consider $f \circ g \co C \rightarrow C.$ It is clearly a self-local equivalence, so ker $f\circ g \subset$ ker $g$. But on the other hand   ker $g\subset$ ker $f\circ g $, so  ker $f\circ g =$ ker $g$. Hence $f|_{\text{im} g}$ is injective. Parallelly $g|_{\text{im} f}$ is injective. Since both im $g$ and im $f$ are finitely generated complex, it follows that $f|_{\text{im} g}$ is indeed an isomorphism of chain complexes.
\end{proof}

As a result, the following is well-defined. 

\begin{definition}[Definition 3.9 in \cite{HHL}]\label{def:conn}
Let $C=\CFKUV(Y,K)$, where $K$ is a knot in an integer homology sphere $Y$ and $f$ a maximal self-local equivalence. The {\em connected complex} is defined to be im $f$, denoted $C_{\mathrm{conn}}$ .
\end{definition}

\begin{proposition}[Lemma 3.6 in \cite{HHL}]\label{pro:acyc1}
Suppose $C_{\mathrm{conn}}$ is the connected complex of $C=\CFKUV(Y,K)$, then $C$ is isomorphic to $C_{\mathrm{conn}} \oplus A$ as chain complex, where $A$ is a subcomplex such that $A\otimes_{\F[\U,\V]} \F[\U,\V,(\U\V)^{-1}]$ has trivial homology.
\end{proposition}

\begin{proof}
Suppose $f$ is a maximal self-local equivalence, then by Lemma \ref{le:max}, $f|_{\mathrm{im} f}$ is injective.  Let $A=\mathrm{ker} f$ and we aim to show $C$ is isomorphic to $\mathrm{im} f\oplus \mathrm{ker} f$  as chain complex. There is a standard algebra argument to prove this.  For $m \in C$, $(n,p) \in  \mathrm{im} f\oplus \mathrm{ker} f$,  define maps in both directions:
\begin{align*}
(n,p) &\longrightarrow p+\big(f|_{\mathrm{im} f}\big)^{-1}(n)\\
m &\longrightarrow (f(m),\big(f|_{\mathrm{im} f}\big)^{-1}\circ f(m)+m).
\end{align*}  
This gives desired isomorphism.  The fact that $A\otimes_{\F[\U,\V]} \F[\U,\V,(\U\V)^{-1}]$ has trivial homology follows from the definition of self-local equivalence. 
\end{proof}

\begin{proposition}[Proposition 3.10 in \cite{HHL}]\label{pro:acyc2}
Let $C=\CFKUV(Y_1,K_1)$ and $C'=\CFKUV(Y_2,K_2)$ be locally equivalent knot complexes, then $C_{\mathrm{conn}}$ and $C'_{\mathrm{conn}}$ are isomorphic as chain complexes.
\end{proposition}

\begin{proof}
Since $C$ is local equivalent to $C'$,  there exists bigrading preserving, $\F[\U,\V]$-equivariant chain maps
\begin{align*}
&F\co C \longrightarrow C' \\
&G\co C' \longrightarrow C
\end{align*}
such that $F$ and $G$ induce isomorphisms on $(\U\V)^{-1}H_*$.

Suppose $f$ and $g$ are maximal self-local equivalence of $C$ and $C'$ respectively. Then $ f \circ G \circ g \circ F$ is a self-local equivalence of $C$. So ker $ G \circ g \circ F \circ f \subset$ ker $f$. But on the other hand,  ker $f \subset $ ker $ G \circ g \circ F \circ f$. We have ker $ G \circ g \circ F \circ f =$ ker $f$. In particular, $g \circ F|_{\mathrm{im} f}\co \mathrm{im} f \rightarrow \mathrm{im} g$ is injective. Parallelly,  $f \circ G|_{\mathrm{im} g}\co \mathrm{im} g \rightarrow \mathrm{im} f$ is injective.
Since both im $g$ and im $f$ are finitely generated, $g \circ F|_{\mathrm{im} f}$ and $f \circ G|_{\mathrm{im} g}$ are isomorhisms of chain complexes.
\end{proof}

\section{Linear Independence of $(Y_n,K_n)$}

In this section, we prove the linear independence of the previously constructed pairs $(Y_n,K_n)$  in the group $\CZhat / \CZ$ using connected knot complex. Recall the definition of pairs $(Y_n,K_n)$: For $n>1$, let $M_n$ denote $+1$-surgery on $T_{2,4n-1}$, the $(2,4n-1)$-torus knot, let $Y_n=M_n\conn -M_n$, and let $K_n\subset Y_n$ denote the connected sum of the core of the surgery in $M_n$ with the unknot in $-M_n$. We remind the reader we use $\simeq$ to denote  local equivalence, and since throughout the section the computation is carried out in the group $\J$, we will use additive notation for tensor product of chain complexes over the ring $\F[\U,\V]$.

 Let $C_n$ denote the connected knot complex of $\CFKUV(-Y_n,-K_n)$, and $C_n^{*}$  the connected knot complex of  $\CFKUV(Y_n,K_n)$. Recall $C_n^*$ is isomorphic to $ (C_n)^*$ as a chain complex.
 
 We recapitulate the properties regarding pairs $(Y_n,K_n)$ that we have studied over the preivous sections (See Theorem \ref{thm:Lspace}), using the language of the chosen ring $\F[\U,\V]$ for the knot Floer complex.
\begin{proposition} \label{prop:rephrase}
The connected knot complex $C_n$, where $n\geqslant 1$, is generated by $x_0, x_1$ and $y_1$ with bigradings 
\begin{align*}
&\gr (x_0)=(-2,0) \\
& \gr (x_1)=(0,-2) \\ 
&\gr (y_1)=(-2n+1,-2n+1).
\end{align*}
and the differential
$$\d y_1 = \U^{n-1}\V^{n}x_0 +\U^{n}\V^{n-1}x_1. $$
The connected knot complex $C_n^*$, where $n\geqslant 1$, is generated by $x_0^*, x_1^*$ and $y_1^*$ with bigradings 
\begin{align*}
&\gr (x_0^*)=(2,0) \\
& \gr (x_1^*)=(0,2) \\
&\gr (y_1^*)=(2n-1,2n-1).
\end{align*}
and the differentials
$$\d \U x_0^* = \d \V x_1^* = \U^n\V^ny_1^*.$$
Morever, for $n\geqslant 2$, we also have
\begin{enumerate}
\item{$\varep(C_n)=\varep(C_n^*)=0;$}
\item{$\tau(C_n)=1,\quad \tau(C_n^*)=-1.$}
\end{enumerate}
\end{proposition}

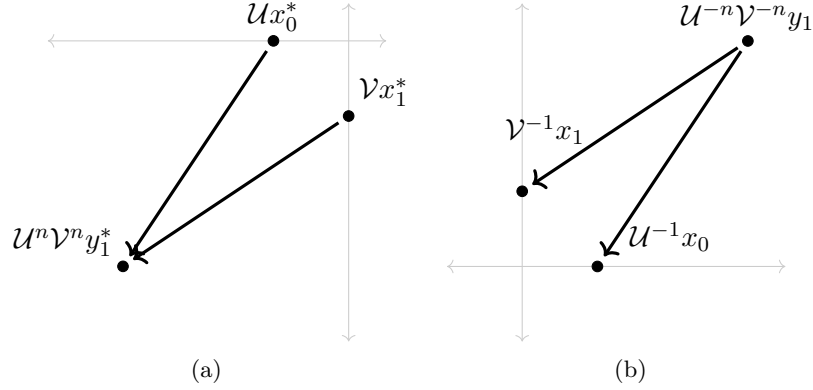
\begin{figure}
\subfigure[]{
\begin{tikzpicture}

	\begin{scope}[thin, black!20!white]
		\draw [<->] (-3.5, 0.5) -- (1, 0.5);
		\draw [<->] (0.5, -3.5) -- (0.5, 1);
	\end{scope}

	\filldraw (-2.5, -2.5) circle (2pt) node[] (a){};
	\filldraw (0.5, -0.5) circle (2pt) node[] (b){};
	\filldraw (-0.5, 0.5) circle (2pt) node[] (c){};	

	\draw [very thick, ->] (b) -- (a);
	\draw [very thick, ->] (c) -- (a);

	\node [above] at (c) {$\U x_0^*$};
	\node [above right] at (b) {$\V x_1^*$};
	\node [above left] at (a) {$\U^n\V^ny_1^*$};

\label{subfig:cn}
\end{tikzpicture}

}
\subfigure[]{
\begin{tikzpicture}

	\begin{scope}[thin, black!20!white]
		\draw [<->] (-0.5, 0.5) -- (4, 0.5);
		\draw [<->] (0.5, -0.5) -- (0.5, 4);
	\end{scope}

	\filldraw (0.5, 1.5) circle (2pt) node[] (a){};
	\filldraw (1.5, 0.5) circle (2pt) node[] (b){};
	\filldraw (3.5, 3.5) circle (2pt) node[] (c){};	

	\draw [very thick, ->] (c) -- (a);
	\draw [very thick, ->] (c) -- (b);

	\node [above] at (c) {$\U^{-n}\V^{-n}y_1$};
	\node [ right] at (1.8,0.9) {$\U^{-1}x_0$};
	\node [above] at (0.8,2) {$\V^{-1}x_1$};

\label{subfig:cn*}
\end{tikzpicture}

}
\label{fig:complex}
\caption{Figure \ref{subfig:cn} shows $C_n^*$.  Figure \ref{subfig:cn*} shows $C_n$ in the preimage of $\U^n\V^n$. }
\end{figure}

Since the complex $C_1$ and  $C_1^*$ are isomorphic to the knot complex of right and left handed trefoil respectively, for the rest of the section we will only consider  $C_n$ and  $C_n^*$ with $n \geqslant 2$.
Our goal is to prove the following statement, which will in turn finish the proof of Proposition \ref{prop:kn}.
\begin{proposition}\label{prop:linear}
If there exist nonnegative integers $p_n,q_n, N$ , such that 
\begin{equation}\label{equ:1}
\sum_{n=2}^{N}(p_n C_n + q_n C_n^{*}) \simeq C
\end{equation}
 for some knot complex $C$ of a knot in $S^{3}$, then $p_n =q_n $ for all $n \geqslant 2$.  
\end{proposition}

Note that $C_n+C_n^*\simeq \F[\U,\V]$. When not all $p_n=q_n$ we can rewrite Equation (\ref{equ:1}) as
$$ \sum_{i=1}^{k} C_{n_i} + \sum_{i=1}^{l}  C_{m_i}^{*} \simeq C $$

where $2\leqslant n= n_1\leqslant .. \leqslant n_{k}$ , $2\leqslant m=m_1\leqslant .. \leqslant m_{l}$. Without the loss of generality we are assuming $n > m$. A straightforward computation shows that
\begin{align*}
\varep (C) =& \varep \big( \sum_{i=1}^{k} C_{n_i} + \sum_{i=1}^{l} C_{m_i}^{*}\big)=0 \\
\tau (C) =& \tau \big( \sum_{i=1}^{k} C_{n_i} + \sum_{i=1}^{l}  C_{m_i}^{*}\big)=k-l
 \end{align*}
 If $l\neq k$ this contradicts the fact that $C$ is the knot complex of a knot in $S^{3}$. Hence we care only about the case

\begin{equation}\label{equ:linear}
\sum_{i=1}^{k} C_{n_i} + \sum_{i=1}^{k}  C_{m_i}^{*} \simeq C 
\end{equation}
where $2\leqslant n= n_1\leqslant .. \leqslant n_{k}$ , $2\leqslant m=m_1\leqslant .. \leqslant m_{k}$ and $n > m$.

For a chain complex $C$ over $\F[\U,\V]$, let $C_{\mathrm{vert}} := C / (\U=0) \otimes \F[\V,\V^{-1}] $. Define the {\em vertical homology} $H_{\mathrm{vert}}$ to be $H_*(C_{\mathrm{vert}})$.
\begin{proposition}
If $C$ is the knot complex of a knot in $S^{3}$, then $H_{\mathrm{vert}}(C)$ is one-dimensional. If $C$ is locally equivalent to a knot complex of a knot in $S^3$, then  $H_{\mathrm{vert}}(C_{\mathrm{conn}})$ is one-dimensional.
\end{proposition}

\begin{proof}
Given a knot complex $C$ for a knot in $S^{3}$,  $C / (\U=0) \otimes \F[\V,\V^{-1}] $ recovers $\CFh(S^3)$, whose homology is one-dimensional.
If $C$ is locally equivalent to a knot complex of a knot in $S^3$, then $C_{\mathrm{conn}}$ is isomorphic to a connected knot complex of a knot in $S^3$. So $H_{\mathrm{vert}}(C_{\mathrm{conn}})$ is one-dimensional.
\end{proof}

The strategy of our proof for Proposition \ref{prop:linear} is to show Equation (\ref{equ:linear}) never holds true by comparing the connected complex on both sides. Let $C'=\sum_{i=1}^{k} C_{n_i} + \sum_{i=1}^{k} C_{m_i}^{*}$.  By the way of contradiction, if Equation (\ref{equ:linear}) holds true, it would imply that $C'_{\mathrm{conn}}$ is isomorphic to $C_{\mathrm{conn}}$ as a chain complex. In particular, $H_{\mathrm{vert}}(C'_{\mathrm{conn}})=H_{\mathrm{vert}}(C_{\mathrm{conn}})$. Hence it suffices to prove dim $(H_{\mathrm{vert}}(C'_{\mathrm{conn}})) >1$.

\begin{figure}
\begin{tikzpicture}

	\begin{scope}[thin, black!20!white]
		\draw [<->] (-7, 0.5) -- (1, 0.5);
		\draw [<->] (0.5, -7) -- (0.5, 1);
	\end{scope}
	\draw[step=1, black!50!white, very thin] (-6.9, -6.9) grid (0.9, 0.9);

	\filldraw  (-3.3, 0.7) circle (2pt) node[] (a){};
	\filldraw (-2.3, -0.3) circle (2pt) node[] (b){};
	\filldraw (-1.3, -1.3) circle (2pt) node[] (c){};	
  \filldraw (-0.3, -2.3) circle (2pt) node[] (d){};	
 	\filldraw (0.7, -3.3) circle (2pt) node[] (e){};	
  \filldraw (-4.3, -1.3) circle (2pt) node[] (f){};
  \filldraw (-3.3, -2.3) circle (2pt) node[] (g){};
   \filldraw (-2.3, -3.3) circle (2pt) node[] (h){};
   \filldraw (-1.3, -4.3) circle (2pt) node[] (i){};

\filldraw[gray!60] (-2.7, -0.7) circle (2pt) node[] (j){};
\filldraw[gray!60] (-1.7, -1.7) circle (2pt) node[] (k){};
\filldraw[gray!60] (-0.7, -2.7) circle (2pt) node[] (l){};
\filldraw[gray!60] (-3.7, -2.7) circle (2pt) node[] (m){};
\filldraw[gray!60] (-2.7, -3.7) circle (2pt) node[] (n){};
\filldraw[gray!60] (-5.5, -2.5) circle (2pt) node[] (o){};
\filldraw[gray!60] (-4.5, -3.5) circle (2pt) node[] (p){};
\filldraw[gray!60] (-3.5, -4.5) circle (2pt) node[] (q){};
\filldraw[gray!60] (-2.5, -5.5) circle (2pt) node[] (r){};
\filldraw[gray!60] (-6.5, -4.5) circle (2pt) node[] (s){};
\filldraw[gray!60] (-5.5, -5.5) circle (2pt) node[] (t){};
\filldraw[gray!60] (-4.5, -6.5) circle (2pt) node[] (u){};

	\draw [very thick, -] (a) -- (f);
	\draw [very thick, -] (b) -- (f);
\draw [very thick, -] (b) -- (g);
\draw [very thick, -] (c) -- (g);
\draw [very thick, -] (c) -- (h);
\draw [very thick, -] (d) -- (h);
\draw [very thick, -] (d) -- (i);
\draw [very thick, -] (e) -- (i);

\draw [ thick, -] (j) -- (m);
\draw [ thick, -] (k) -- (m);
\draw [ thick, -] (k) -- (n);
\draw [ thick, -] (l) -- (n);
\draw [ thick, -] (j) -- (o);
\draw [ thick, -] (k) -- (p);
\draw [ thick, -] (l) -- (q);
\draw [ thick, -] (l) -- (r);
\draw [ thick, -] (m) -- (s);
\draw [ thick, -] (n) -- (t);
\draw [ thick, -] (o) -- (s);
\draw [ thick, -] (p) -- (s);
\draw [ thick, -] (p) -- (t);
\draw [ thick, -] (q) -- (t);
\draw [ thick, -] (q) -- (u);
\draw [ thick, -] (r) -- (u);
\draw [ thick, -] (l) -- (i);

\end{tikzpicture}

\label{fig:2223*}
\caption{Complex $3C^*_2 +C^*_3$, which contains the inverse saw-edge $\C^*(4,2)$ as a subcomplex (generators of $\C^*(4,2)$ are depicted in darker color).   }
\end{figure}
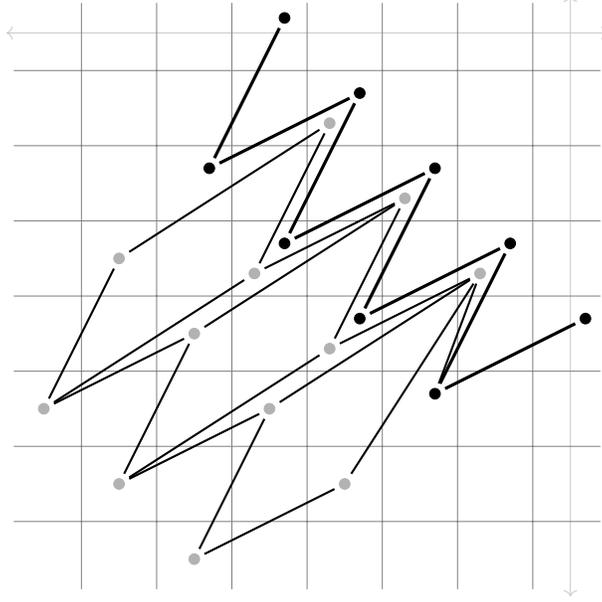

As an effort to discuss the objects at hand precisely, we introduce the following definitions.

\begin{definition}
Define a {\em saw-edge} $\C(k,n)$ of length $k$ with {\em tooth size} $n$, where $k \geqslant 1$ and $n\geqslant 2$,  to be a chain complex over $\F[\U,\V]$, generated by $x_i^{(k)},  i=0,...,k$ and $y_i^{(k)},  i=1,...,k.$

Bigradings are given by
\begin{align*}
&\gr(x^{(k)}_i)=(-2(k-i),-2i)\\ 
&\gr(y^{(k)}_i)=(-2(n+k-1)+1,-2(n-1+i)+1).
 \end{align*}
 while the differentials are 
$$\d y^{(k)}_i=\U^{n}\V^{n-1}x^{(k)}_i+\U^{n-1}\V^n x^{(k)}_{i-1}. \qquad \qquad \qquad$$

Define an {\em inverse saw-edge} $\C^*(k,n)$ of length $k$ with {\em tooth size} $n$ , where $k \geqslant 1$ and $n\geqslant 2$,  to be simply the dual complex of the saw-edge  $\C(k,n)$. Dualizing the generators of $\C(k,n)$, the inverse saw-edge $\C^*(k,n)$ is generated by 
 $x_i^{(k),*},  i=0,...,k$ and $y_i^{(k),*},  i=1,...,k.$

 Bigradings are given by
 \begin{align*}
  &\gr(x^{(k),*}_i)=(2(k-i),2i)\\
  &  \gr(y^{(k),*}_i)=(2(n+k-1)-1,2(n-1+i)-1).
  \end{align*}
  while the differentials are 
\begin{equation*}
\d x^{(k),*}_i = 
\begin{cases}
\U^{n-1}\V^ny^{(k),*}_1, & \quad i=0, \\
\U^{n}\V^{n-1}y^{(k),*}_i + \U^{n-1}\V^{n}y^{(k),*}_{i+1}, & \quad i=1,...k-1,\\
\U^n\V^{n-1}y^{(k),*}_k & \quad i=k.
\end{cases}
\end{equation*}
\end{definition}

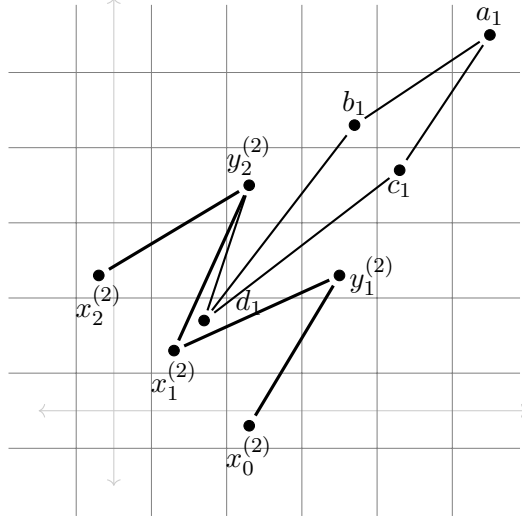
\begin{figure}
\begin{tikzpicture}

	\begin{scope}[thin, black!20!white]
		\draw [<->] (-0.5, 0.5) -- (6, 0.5);
		\draw [<->] (0.5, -0.5) -- (0.5, 6);
	\end{scope}
	\draw[step=1, black!50!white, very thin] (-0.9, -0.9) grid (5.9, 5.9);
	
	\filldraw (0.3, 2.3) circle (2pt) node[] (a){};
	\filldraw (1.3, 1.3) circle (2pt) node[] (b){};
	\filldraw (2.3, 0.3) circle (2pt) node[] (c){};
	\filldraw (2.3, 3.5) circle (2pt) node[] (d){};
	\filldraw (3.5, 2.3) circle (2pt) node[] (e){};
	
		\filldraw (1.7, 1.7) circle (2pt) node[] (f){};
	  \filldraw (3.7, 4.3) circle (2pt) node[] (g){};
		\filldraw (4.3, 3.7) circle (2pt) node[] (h){};
		\filldraw (5.5, 5.5) circle (2pt) node[] (i){};

	\draw [very thick, -] (d) -- (a);
	\draw [very thick, -] (d) -- (b);
	\draw [very thick, -] (e) -- (b);
	\draw [very thick, -] (e) -- (c);
	
	\draw [ thick, -] (g) -- (f);	
	\draw [ thick, -] (h) -- (f);
	\draw [ thick, -] (i) -- (g);	
	\draw [ thick, -] (i) -- (h);
	\draw [ thick, -] (d) -- (f);

\node [below] at (a) {$x^{(2)}_2$};
\node [below] at (b) {$x^{(2)}_1$};
\node [below] at (c) {$x^{(2)}_0$};
\node [above] at (d) {$y^{(2)}_2$};
\node [right] at (e) {$y^{(2)}_1$};
\node  at (2.3,1.94) {$d_1$};
\node [above] at (g) {$b_1$};
\node [below] at (h) {$c_1$};
\node [above] at (i) {$a_1$};

\end{tikzpicture}
\label{fig:23}
\caption{ Complex $C_2 + C_3$, where generators are labelled abstractly, missing their actual $\U,\V$ decorations.  Using the notations from Proposition \ref{prop:cnsum}, $a_1,b_1,c_1,d_1$ generate subcomplex $\D$, where $b_1$ plays the role of $w_2$. }
\end{figure}

Next,  we want to study some key properties of the complex $\sum_{i=1}^{k} C_{n_i}$. Roughly speaking, $\sum_{i=1}^{k} C_{n_i}$ consists of two parts, the saw-edge $\C(k,n)$ and a subcomplex we call $\D$. We can pick a basis for $\sum_{i=1}^{k} C_{n_i}$ over $\F[\U,\V]$, such that on the module level, $\sum_{i=1}^{k} C_{n_i}$ decomposites as a direct sum   $\C(k,n)\oplus \D$, where $\C(k,n)$ is generated by $x^{(k)}_i, i=0,...,k$ and $y^{(k)}_i, i=1,...,k$, whereas $\D$ is generated by $w_i , i=1,...,k$ and other generators. While as a chain complex, the saw-edge $\C(k,n)$ is no longer a direct summand, there are generators we denote by $w_i , i=1,...,k$ in $\D$, such that the image of $\d \big((\U\V)^{h_i}y^{(k)}_i + w_i \big)$ is inside $\C(k,n)$. More precisely, we have 
$\d \big((\U\V)^{h_i}y^{(k)}_i + w_i \big)= (\U\V)^{h_i} \big(\U^n\V^{n-1}x^{(k)}_i + \U^{n-1}\V^n x^{(k)}_{i-1}\big), i=1,...,k$ 
where $h_i$ is a nonnegative integer. We will slightly abuse the notation and write $\d \big(y^{(k)}_i + (\U\V)^{-h_i}w_i \big)=\U^n\V^{n-1}x^{(k)}_i + \U^{n-1}\V^n x^{(k)}_{i-1}, i=1,...,k$ instead.

\begin{proposition}\label{prop:cnsum}
There exists a basis for the complex $\sum_{i=1}^{k} C_{n_i}$, where $2\leqslant n=n_1\leqslant .. \leqslant n_{k}$, and a subcomplex $\D$, such that the following are satisfied:
\begin{enumerate}
\item{on the module level, $\sum_{i=1}^{k} C_{n_i}$ is isomorphic to $\C(k,n)\oplus \D$ as a module;}\label{it:1}
\item{$\sum_{i=1}^{k} C_{n_i}/\D$ is isomorphic to $\C(k,n)$ as a chain complex;}\label{it:2}
\item each {$x^{(k)}_i \in \C(k,n), i=0,...,k$ generates $(\U\V)^{-1}H_* \big(\sum_{i=1}^{k} C_{n_i} \big);$}\label{it:3}
\item{for each $y^{(k)}_i \in \C(k,n), i=1,...,k$, there exists some $w_i \in \D$, such that 
$$\d \big(y^{(k)}_i + (\U\V)^{-h_i}w_i \big)=\U^n\V^{n-1}x^{(k)}_i + \U^{n-1}\V^n x^{(k)}_{i-1},$$
where $h_i$ is a nonnegative integer, and $\d$ is the differential of the chain complex $\sum_{i=1}^{k} C_{n_i}$.}\label{it:4}
\end{enumerate}
\end{proposition}

Note that in (\ref {it:3}),(\ref {it:4}) above when we say $x^{(k)}_i , y^{(k)}_i  \in \C(k,n)$ it specifically means that we see  $\C(k,n)$ as a submodule of $\sum_{i=1}^{k} C_{n_i}$, take elements $x^{(k)}_i , y^{(k)}_i$ from $\C(k,n)$ and equip them  with the differential of the whole complex $\sum_{i=1}^{k} C_{n_i}$.
In order to prove Proposition \ref{prop:cnsum} we require a lemma. In the following lemma, we will construct a basis  $x^{(k+1)}_i, i=0,...,k+1, y^{(k+1)}_i, i=1,...,k+1$ and $\{ a_i, b_i, c_i, d_i \}_{i=1,...,k}$ that generates $C'':=\C(k,n)+C_\ell$ over $\F[\U,\V],$  where $ \ell \geqslant n$. On the module level, $C''$ decomposites as a direct sum $\C(k+1,n) \oplus \D_0$, where $\C(k+1,n)$ is generated by $x^{(k+1)}_i, i=0,...,k+1$ and  $y^{(k+1)}_i, i=1,...,k+1$, while $\D_0$ is generated by $\{ a_i, b_i, c_i, d_i \}_{i=1,...,k}$.

\begin{lemma}\label{le:kcn}
There exists a basis for the complex $C'':=\C(k,n)+C_\ell,$ where $ \ell \geqslant n$, and a subcomplex $\D_0$, such that
\begin{enumerate}
\item{on the module level, $C''=\C(k+1,n) \oplus \D_0$ as a module;}\label{it:1'}
\item{$C''/\D_0$ is isomorphic to $\C(k+1,n)$ as a chain complex;}\label{it:2'}
\item each {$x^{(k+1)}_i \in \C(k+1,n), i=0,...,k+1$ generates $(\U\V)^{-1}H_* \big(C'' \big);$}\label{it:3'}
\item{there exists $w_{k+1} \in \D_0$ such that
$$\d \big(y^{(k+1)}_{k+1} + (\U\V)^{-\ell+n}w_{k+1} \big)=\U^n\V^{n-1}x^{(k+1)}_{k+1}+ \U^{n-1}\V^n x^{(k+1)}_{k},$$
and we have
$$\d y^{(k+1)}_i =\U^n\V^{n-1}x^{(k)}_i + \U^{n-1}\V^n x^{(k)}_{i-1},  i=1,...,k $$
where $\d$ is the differential of the chain complex $C''$.
}\label{it:4'}
\end{enumerate}
\end{lemma}

Here $x^{(k+1)}_i , y^{(k+1)}_i $ are all elements of $\C(k+1,n)$, seen as a submodule of $C''$, equipped with the differential of the whole complex $C''$.

\begin{proof}[Proof of Lemma \ref{le:kcn}]
Suppose complex $\C(k,n)$ is generated by $x_i^{(k)},  i=0,...,k$ and $y_i^{(k)},  i=1,...,k.$ with differentials
$\d y^{(k)}_i=\U^{n}\V^{n-1}x^{(k)}_i+\U^{n-1}\V^n x^{(k)}_{i-1},$ and  complex $C_\ell$ is generated by $x^{(1)}_0, x^{(1)}_1$ and $y^{(1)}_1$ with differential
$\d y_1^{(1)} = \U^{\ell-1}\V^{\ell}x_0^{(1)} +\U^{\ell}\V^{\ell-1}x_1^{(1)}. $ We will construct a basis for $C''=\C(k,n)+C_\ell.$
Set

\begin{align*}
& x^{(k+1)}_i=
\begin{cases}
x^{(k)}_k \otimes x^{(1)}_1 , \quad & i=k+1,\\
x^{(k)}_i \otimes x^{(1)}_0, \quad & i=0,1,...,k;
\end{cases}
\\
& y^{(k+1)}_i=
\begin{cases}
y^{(k)}_k \otimes x^{(1)}_1,  \quad &i=k+1,\\
y^{(k)}_i \otimes x^{(1)}_0 ,  \quad &i=1,...,k;
\end{cases}
\\
& a_i=y^{(k)}_i \otimes y^{(1)}_1,  \quad i=1,...,k;\\
& b_i=
\begin{cases}
 x^{(k)}_k \otimes y^{(1)}_1 + (\U\V)^{\ell-n}y^{(k)}_k \otimes x^{(1)}_1,  \quad & i=k,\\
y^{(k)}_{i+1} \otimes x^{(1)}_0 + y^{(k)}_i \otimes x^{(1)}_1 , \quad & i=1,...,k-1;
\end{cases}
\\
& c_i =  x^{(k)}_{i-1} \otimes y^{(1)}_1 + (\U\V)^{\ell-n}y^{(k)}_i \otimes x^{(1)}_0, i=1,...,k;\\
& d_i = x^{(k)}_i \otimes x^{(1)}_0 + x^{(k)}_{i-1} \otimes x^{(1)}_1, i=1,...,k.
\end{align*}

Differentials are
\begin{align*}
&  \d y^{(k+1)}_i =
\begin{cases}
\U^{n}\V^{n-1} x^{(k+1)}_{k+1} + \U^{n-1}\V^n x^{(k+1)}_k + \U^{n-1}\V^n d_k, \quad & i=k+1,\\
\U^{n}\V^{n-1} x^{(k+1)}_i + \U^{n-1}\V^n x^{(k+1)}_{i-1}, \quad & i=1,...,k;
\end{cases}
\\
& \d a_i=
\begin{cases}
\U^n\V^{n-1}b_k + \U^{n-1}\V^n c_k,   \quad & i=k,\\
\U^\ell\V^{\ell-1}b_i + \U^{n}\V^{n-1} c_{i+1}+\U^{n-1}\V^n c_i,   \quad & i=1,...,k-1;
\end{cases}
\\
& \d b_i=
\begin{cases}
\U^{\ell-1}\V^\ell d_k, \quad & i=k,\\
\U^n \V^{n-1} d_{i+1} +\U^{n-1} \V^n d_{i}, \quad & i=1,...,k-1;
\end{cases}
\\
& \d c_i = \U^\ell \V^{\ell-1} d_i,  \quad i=1,...,k.
\end{align*}
Let $\D_0$ be the subcomplex generated by $\{ a_i, b_i, c_i, d_i \}_{i=1,...,k}$. Conditions (\ref{it:1'}),(\ref{it:2'}) and (\ref{it:3'}) are all straightforward to verify. To show (\ref{it:4'}),  let $w_{k+1}=b_k$, it then follows that $\d \big(y^{(k+1)}_{k+1} + (\U\V)^{-\ell+n}w_{k+1} \big)=\U^n\V^{n-1}x^{(k+1)}_{k+1}+ \U^{n-1}\V^n x^{(k+1)}_{k}.$ This completes the proof of Lemma \ref{le:kcn}.
\end{proof}

\begin{remark}
In Lemma \ref{le:kcn} if $\ell=n$, $\D_0$ is actually a direct summand and $\C(k,n)+C_n \simeq \C(k+1,n).$
\end{remark}

\begin{proof}[Proof of Proposition \ref{prop:cnsum}]
We will prove the statement by induction over $k$. $C_n \simeq \C(1,n)$ follows directly from Proposition \ref{prop:rephrase}.

Suppose the proposition holds for $\sum_{i=1}^{k} C_{n_i}$, where $\sum_{i=1}^{k} C_{n_i}=\C(k,n)\oplus \D$ as a module. Consider $\widetilde{C}:=\sum_{i=1}^{k} C_{n_i}+C_\ell,$ where $2\leqslant n=n_1\leqslant .. \leqslant n_{k}\leqslant \ell.$ From Lemma \ref{le:kcn} we know that on the module level $\big(\C(k,n)\oplus \D\big)\otimes C_\ell=\big(\C(k,n)\otimes C_\ell\big)\oplus \big(\D\otimes C_\ell \big)=\C(k+1,n)\oplus \D_0 \oplus \big(\D\otimes C_\ell \big)$. Note that this gives a basis of  $\widetilde{C}$ that extends the basis of $C''$, which has been chosen in the proof of Lemma \ref{le:kcn}.

Define $\D_1:= \D + C_\ell  $ as chain complexes. Note that $\D_1$ has underlying module $\D \otimes C_\ell  $, which appeared in the above expression, now further equipped with the subcomplex differential. 

Define $\widetilde{\D}:=\D_0 \oplus \D_1  $ first as a module, then equip it with the differential of the chain complex $\widetilde{C}$. Observe that $\widetilde{\D}$ is a subcomplex of $\widetilde{C}$ because both $\D_0$ and $\D_1$  are subcomplexes. We will prove that the complex $\widetilde{C}$ together with the subcomplex $\widetilde{\D}$ under the basis chosen above satisfy all $4$ requirements in Proposition \ref{prop:cnsum}.

Obviously, $\widetilde{C}=\C(k+1,n) \oplus \widetilde{\D} $ as a module. This completes (\ref{it:1}).

To prove (\ref{it:2}), observe that $\widetilde{C}/\widetilde{\D}$ is isomorphic to $\big( \widetilde{C}/\D_1 \big)/\D_0$  as a chain complex. By definition, $\widetilde{C}/\D_1=\big(\sum_{i=1}^{k} C_{n_i}+C_\ell\big)/(\D+C_\ell)$. The natural map between $\big(\sum_{i=1}^{k} C_{n_i}+C_\ell\big)/(\D+C_\ell)$ and $\sum_{i=1}^{k} C_{n_i}/\D+C_\ell$ is a chain complex isomorphism. By induction hypothesis $\sum_{i=1}^{k} C_{n_i}/\D$ is isomorphic to $\C(k,n)$ and finally according to Lemma \ref{le:kcn}, $\big( \C(k,n)+C_\ell \big)/\D_0$ is isomorphic to $\C(k+1,n)$ as a chain complex, finishing the proof of (\ref{it:2}).

From our construction of $\C(k+1,n)$ it is not hard to see
\begin{align*}
x^{(k+1)}_i=
\begin{cases}
x^{(k)}_k \otimes x^{(1)}_1 , \quad & i=k+1,\\
x^{(k)}_i \otimes x^{(1)}_0, \quad & i=0,1,...,k;
\end{cases}
\end{align*}
generates $(\U\V)^{-1}H_* \big(\widetilde{C}\big)$, proving (\ref{it:3}).

In order to show (\ref{it:4}), suppose there exists some $w_i\in \D$ and nonnegative $s_i$ for each $i$ such that 
$$\d \big(y^{(k)}_i + (\U\V)^{-h_i}w_i \big)=\U^n\V^{n-1}x^{(k)}_i + \U^{n-1}\V^n x^{(k)}_{i-1}.$$
For $i=1,...,k,$
\begin{align*}
\d \big(y^{(k+1)}_i + (\U\V)^{-h_i}w_i \otimes x^{(1)}_0 \big)&=\U^n\V^{n-1}x^{(k)}_i\otimes x^{(1)}_0 + \U^{n-1}\V^n x^{(k)}_{i-1}\otimes x^{(1)}_0\\
&=\U^n\V^{n-1}x^{(k+1)}_i +\U^{n-1}\V^{n}x^{(k+1)}_{i-1}
\end{align*}
 If $i=k+1$, then
\begin{align*}
\d \big(y^{(k+1)}_{k+1} + (\U\V)^{-h_i}w_i \otimes x^{(1)}_1  + (\U\V)^{-(l-n)}b_k \big)&=\U^n\V^{n-1}x^{(k)}_i\otimes x^{(1)}_1 + \U^{n-1}\V^n x^{(k)}_{k}\otimes x^{(1)}_0\\
&=\U^n\V^{n-1}x^{(k+1)}_{k+1} +\U^{n-1}\V^{n}x^{(k+1)}_{k}
\end{align*} 
satisfying the requirement of (\ref{it:4}).
\end{proof}

\begin{proposition}
$\C^*(k,m)$ is a subcomplex of $\sum_{i=1}^{k}  C_{m_i}^{*},$ where  $2\leqslant m=m_1\leqslant .. \leqslant m_{k}.$ In particular,  $\sum^k_{i=0} \U^{k-i}\V^i x^{(k),*}_i$ is a generator of $(\U\V)^{-1}H_*$.
\end{proposition}

\begin{proof}
Simply take the dual complex of $\sum_{i=1}^{k}  C_{m_i}.$  Since there is no incoming arrow to $\C(k,m)$, $\C^*(k,m)$ is a subcomplex of $\sum_{i=1}^{k}  C_{m_i}^{*}.$ It is also easy to verify that  $\sum^k_{i=0} \U^{k-i}\V^i x^{(k),*}_i$ generates $(\U\V)^{-1}H_*$.
\end{proof}

\begin{proof}[Proof of Proposition \ref{prop:linear}]

Consider the complex $C'=\sum_{i=1}^{k} C_{n_i} + \sum_{i=1}^{k} C_{m_i}^{*}$. We fix a basis for $\sum_{i=1}^{k} C_{n_i}$  according to Proposition \ref{prop:cnsum}. Similarly fix a basis for $\sum_{i=1}^{k} C_{m_i}$ and take the dual basis. This gives a basis for $C'$. We will suppress the $(k)$ from upper indices since both saw-edge complexes have length $k$.

Suppose $f$ is a maximal self-local equivalence of $C'$. In particular $f$ is $\U,\V$-equivariant, preserves bigradings and maps one generator of $(\U \V)^{-1}H_*$ to another. Observe that there is no vertical arrow in $C'$, therefore all generators in $C'_{\mathrm{conn}}$ survive into $H_{\mathrm{vert}}(C'_{\mathrm{conn}})$. We only need to show $C'_{\mathrm{conn}}=$ im$f$ contains more than one generators of homology.

Recall that for each $y_i \in \sum_{i=1}^{k} C_{n_i}, i=1,...,k$, there exists some $w_i \in \sum_{i=1}^{k} C_{n_i}$, such that 
$$\d \big(y_i + (\U\V)^{-h_i}w_i \big)=\U^n\V^{n-1}x_i + \U^{n-1}\V^n x_{i-1},$$
where $h_i$ is a nonnegative integer.
Define
$$\widetilde{y}_i := w_i + (\U\V)^{h_i}y_i ,$$
and let 
$h:=$ max $\{ h_i \co i=1,...,k \}.$

 Set
\begin{align*}
&\alpha = \sum^k_{i=1} (\U\V)^{h-h_{i}} \widetilde{y_i} \otimes y_i^* + (\U\V)^{h+n-m} \sum^k_{i=0} x_i \otimes x^*_i , \\
& \beta = \sum^k_{i=0} \U^{k-i}\V^i x_0 \otimes x^*_i.
\end{align*}

We want to point out that  $\alpha$ is an actual generator of the chain complex, meaning $\alpha$ contains a non-trivial component with no $\U,\V$ decoration. This fact turns out to be necessary in order for the proof to work. On the other hand, $\beta$ is not an actual generator of the chain complex. 

We compute the bigradings of $\alpha$ and $\beta$:
\begin{align*}
& \gr (\alpha)=(-2(m+h-n),-2(m+h-n)),\\
& \gr (\beta)= \gr (x_0)= (-2k, 0).
\end{align*}
  Since tensoring the complex $C'$ by $C_m+C_m^*$ doesn't change its local equivalence type, we can assume that $k$ is large enough such that $-2k<-2(m+h-n)$.

In fact, $\alpha$ and $\beta$ are both generators of $(\U \V)^{-1}H_*$.  Here we slightly abuse the notation: when we refer to an element in $C'$ we actually mean its counterpart in $C'\otimes _{\F[\U,\V]}\F[\U,\V,(\U\V)^{-1}].$
It is a simple computation that both $\alpha$ and $\beta$ are in ker $\d$.  To show they are not in the image, first assume $\d \big( \sum z\otimes r \big)=\beta$, where each $z \in \sum_{i=1}^{k} C_{n_i}$ and each $r \in  \sum_{i=1}^{k} C_{m_i}^{*}$. Non-trivial components of $\d \big( \sum z\otimes r \big)$ are either of the form $(\d z) \otimes r $ or $z \otimes (\d r)$. Consider the term $\U^{k-i}\V^i x_0 \otimes x^*_i$ in $\beta$.  Since each $x_i$ is in ker $\d$, $x_i^*$ never occurs in the image of $\d$ as a non-trivial component. Hence $\U^{k-i}\V^i x_0 \otimes x^*_i$ is of the form $\sum (\d z)\otimes r$. It follows that $x_0$ has to be in the image, a contradiction. Supposing $\d \big( \sum z\otimes r \big)=\alpha$, similarly,  each term $(\U\V)^{h+n-m} \sum^k_{i=0} x_i \otimes x^*_i$ in $\alpha$ is of the form $\sum (\d z)\otimes r$. Since  $ y_i$ never occurs in the image as a non-trivial component, each term  $(\U\V)^{h-h_{i}} \widetilde{y_i} \otimes y_i^*$ is of the form $\sum  z\otimes (\d r).$  A brief examination shows that the only suitable elements for  $ z\otimes r$ are $\widetilde{y_i} \otimes x_j^*$. However,  there does not exist any combination of $\d \big( \widetilde{y_i} \otimes x_j^* \big)$ equal to $\alpha$. This proves that both $\alpha$ and $\beta$ generate $(\U \V)^{-1}H_*$. 

According to the property that self-local equivalence $f$ maps one generator of $(\U \V)^{-1}H_*$ to another, we know $f(\alpha)$ and $f(\beta)$ are both non-trivial in $C'_{\mathrm{conn}}$. However, since $f$ preserves bigradings, $f(\alpha)\neq f(\beta)$. Moreover, there does not exist any positive integers $s,t$ such that $f(\alpha)=\U^s\V^t f(\beta)$ or $f(\beta)=\U^s\V^t f(\alpha)$. We conclude that $C'_{\mathrm{conn}}$ contains at least $2$ generators of homology.

\end{proof}

\bibliographystyle{amsalpha}
\bibliography{paper}

\end{document}